\documentclass[11pt,letterpaper,reqno]{amsart}

\usepackage[T1]{fontenc}

\usepackage{amsmath}
\usepackage{amssymb}
\usepackage{amsfonts}
\usepackage{amsthm}
\usepackage{bbm}

\usepackage{enumitem}
\usepackage{booktabs}

\usepackage{graphicx}
\usepackage{float}
\usepackage{tikz}
\usepackage{pgfplots}
\pgfplotsset{compat=1.18}
\usetikzlibrary{positioning,shapes.geometric,arrows.meta}

\usepackage{xcolor}
\definecolor{darkgreen}{HTML}{006400}
\usepackage[
colorlinks=true,
linkcolor=blue,
citecolor=green,
urlcolor=magenta,
pdfstartview=FitH
]{hyperref}
\usepackage{doi}

\newtheoremstyle{plainnobrackets}
  {\topsep}{\topsep}{\itshape}{} {\bfseries}{.}{0.5em}
  {\thmname{#1}\thmnumber{ #2}\thmnote{ #3}}
\newtheoremstyle{definitionnobrackets}
  {\topsep}{\topsep}{\normalfont}{} {\bfseries}{.}{0.5em}
  {\thmname{#1}\thmnumber{ #2}\thmnote{ #3}}
\newtheoremstyle{remarknobrackets}
  {\topsep}{\topsep}{\normalfont}{} {\itshape}{.}{0.5em}
  {\thmname{#1}\thmnumber{ #2}\thmnote{ #3}}

\theoremstyle{plainnobrackets}
\newtheorem{thm}{Theorem}[section]
\newtheorem{lem}[thm]{Lemma}
\newtheorem{prop}[thm]{Proposition}
\newtheorem{cor}[thm]{Corollary}

\newtheorem{prob}[thm]{Problem}

\newtheorem{conj}[thm]{Conjecture}

\theoremstyle{definitionnobrackets}
\newtheorem{eg}[thm]{Example}
\newtheorem{defin}[thm]{Definition}

\theoremstyle{remarknobrackets}
\newtheorem{rem}[thm]{Remark}

\numberwithin{equation}{section}

\newcommand{\sa}{\mathrm{sa}}

\newcommand{\bC}{\mathbb{C}}

\newcommand{\bM}{\mathbb{M}}

\newcommand{\bR}{\mathbb{R}}

\newcommand{\sym}{\mathrm{sym}}
\newcommand{\qsym}{\mathrm{qsym}}
\newcommand{\trS}{\mathcal{S}}

\newcommand{\trE}{\mathcal{E}}
\newcommand{\eps}{\varepsilon}
\newcommand{\absop}[1]{\left|#1\right|}
\DeclareMathOperator{\Tr}{Tr}
\DeclareMathOperator{\Ree}{Re}
\DeclareMathOperator{\Imm}{Im}

\newcommand{\tr}{\operatorname{Tr}}

\newcommand{\absv}[1]{\left|#1\right|}

\newcommand{\qsma}[1]{\left|#1\right|_{\mathrm{qsym}}}
\begin{document}

\title[Matrix triangle inequalities]{Matrix Triangle Inequalities: Past, Present, and Future}

\author[J.-C.~Bourin]{Jean-Christophe Bourin}
\address{Universit\'e Marie et Louis Pasteur, CNRS, LmB (UMR 6623), F-25000 Besan\c{c}on, France}
\email{jcbourin@univ-fcomte.fr}

\author[T.~Zhang]{Teng Zhang}
\address{School of Mathematics and Statistics, Xi'an Jiaotong University, Xi'an 710049, P.~R.~China}
\email{teng.zhang@stu.xjtu.edu.cn}


\thanks{The second author is supported by the China Scholarship Council, the Young Elite Scientists Sponsorship Program for PhD Students of the China Association for Science and Technology, and the Fundamental Research Funds for the Central Universities at Xi'an Jiaotong University (Grant No.~xzy022024045).}

\begin{abstract}
We revisit several classical matrix inequalities related to the celebrated Thompson triangle inequality. We then present the second author's version of Thompson's inequality and related results. Finally, we survey a number of recent and elegant matrix triangle inequalities that may be viewed as cousins of Thompson's inequality. Taken together, these results illustrate the central role played by various notions of modulus in the noncommutative setting. Their interplay opens new perspectives in matrix analysis and operator theory.
\end{abstract}

\maketitle
\begingroup
\hypersetup{linkcolor=darkgreen}
\tableofcontents
\endgroup

\section{Introduction}
Within the vast literature on matrix and operator inequalities, one may distinguish two special classes illustrating how matrix analysis can be viewed as a noncommutative counterpart of classical analysis. The first class transforms scalar identities into matrix inequalities, whereas the second provides matrix analogues of scalar inequalities. We briefly review a few representative examples in this introductory section.

Let $\bM_n$ denote the algebra of complex $n\times n$ matrices (also regarded as operators on an $n$-dimensional Hilbert space) , let $\bM_n^+$ be its positive semidefinite cone, and let $\bM_n^{\sa}$ be its Hermitian part.  A celebrated example of the first class is the Golden--Thompson inequality \cite{Gol65,Tho65},
\[
\Tr e^{A+B}\le \Tr(e^Ae^B),
\qquad A,B\in\bM_n^{\sa}.
\]

Turning to the second class, one of the earliest examples is the trace inequality for convex functions. Let  $g(t)$ be a convex function defined on an interval containing the spectra of two Hermitian matrices $A$ and $B$ in $\bM_n$. Then
\begin{equation}\label{vNconv}
{\mathrm{Tr}\,} g\left(\frac{A+B}{2}\right) \le \frac{{\mathrm{Tr}\,} g(A) +{\mathrm{Tr}\,} g(B)}{2}.
\end{equation}
With $g(t)=t\log t$, this brings the concavity of von Neumann's entropy \cite{Neu27}, and thus we call it von Neumann's convexity inequality, although inequalities of this type were likely known earlier. Another famous example in our second class is   Thompson's triangle inequality \cite{Tho76}. 
This theorem states that if $A,B\in\bM_n$, then there exist  unitary matrices $U,V\in\bM_n$ such that
\begin{equation}\label{TI}
|A+B| \le U|A|U^*+V|B|V^*,
\end{equation}
where $|A|=(A^*A)^{1/2}$ denotes the right modulus of $A$.  One of Thompson's motivations was to provide a simplified approach to the  Rotfel'd trace inequality \cite{Rot69}, which states
  that, for any concave function $f:[0,\infty)\to [0,\infty)$,
\begin{equation}\label{Rot}
\mathrm{Tr}\, f(|A+B|) \le \mathrm{Tr}\, f(|A|) +  \mathrm{Tr}\, f(|B|).
\end{equation}
Over the last 25 years, a host of elegant inequalities have been established that extend or complement these two beautiful subadditivity theorems of Rotfel'd \eqref{Rot} and Thompson \eqref{TI}.  

This article, mostly of expository nature,  is concerned with classical results related to  Thompson's inequality, as well as recent symmetrized versions. In  Section~\ref{s:thompson}, we establish   Thompson's inequality \eqref{TI}.  Ky Fan $k$-norms and Schatten $p$-norms are studied in Section~\ref{s:symmetric-norms}, and we show that the triangle inequality for these classical norms can be derived from Thompson's inequality. 
In Section~\ref{s:majorization}, 
 we  develop  a rather original approach of majorization theory,  thereby completing our study of symmetric norms. 
Section~\ref{s:right} examines what remains of Thompson-type estimates when
the unitary conjugations on the right-hand side are removed, and explains why
additional constants or averaging procedures are then unavoidable. We then turn to symmetrized versions of the modulus.
Section~\ref{s:quadratic} is devoted to the quadratic symmetric modulus
\[
|Z|_{\qsym}:=
\left(\frac{|Z|^2+|Z^*|^2}{2}\right)^{1/2}.
\]
A remarkable feature of this modulus is that it still satisfies a
Thompson-type triangle inequality: for every \(A,B\in\bM_n\), there exist
unitary matrices \(U,V\in\bM_n\) such that
\[
|A+B|_{\qsym}
\le
U|A|_{\qsym}U^*+V|B|_{\qsym}V^* .
\]
Section~\ref{s:symmetric} studies the arithmetic symmetric modulus
\[
|Z|_{\sym}:=\frac{|Z|+|Z^*|}{2},
\]
for which the direct Thompson-type inequality fails in general, but several
weaker or averaged forms remain available.
Finally, Section~\ref{s:hybrid} collects recent hybrid inequalities which
compare these different moduli and may be viewed as further variations on
Thompson's triangle inequality. Hence the topic of matrix triangle inequalities is, 50 years after Thompson, a very active field of research.

\section{Thompson's inequality}\label{s:thompson}

We begin with Thompson's inequality itself. The proof is based on the Fan--Hoffman lemma, and the same method also gives a useful Cartesian version.

  If $A\in\bM_n$ is  Hermitian, we denote by 
$\lambda_1^{\downarrow}(A)\ge \cdots\ge  \lambda_n^{\downarrow}(A)$  its eigenvalues arranged  in   nonincreasing order, and by
$
A^{\downarrow}$ the diagonal matrix in $\bM_n$ having these eigenvalues on the diagonal.

We  assume the reader is familiar with the min-max principle. This principle shows that, 
 for Hermitian matrices $A,B\in\bM_n$, the eigenvalue order is equivalent to the unitary orbit order:
$\lambda_j^{\downarrow}(A)\le \lambda_j^{\downarrow}(B)$ for all $j=
1,\ldots,n$, if and only if $A\le VBV^*$ for some unitary  $V$; see \cite[p. 74, Exercise III.5.5]{Bha97}. 
 If $A=VBV^*$, we write $A\simeq B$, and say that $A$ and $B$ are unitarily equivalent.

For $Z\in\bM_n$, we denote its real, or Hermitian part by
$$
{\mathrm{Re\,}}Z :=\frac{Z+Z^*}{2}.
$$
Similarly, its imaginary part is also the Hermitian matrix
$$
{\mathrm{Im\,}}Z :=\frac{Z-Z^*}{2i}  ={\mathrm{Re\,}}(-iZ).
$$
Thus $Z={\mathrm{Re\,}}Z + i{\mathrm{Im\,}}Z$ which is called the decomposition in real and imaginary parts, or the Toeplitz decomposition, or the Cartesian decomposition.  The Fan-Hoffman lemma \cite{FH55} states:

\vskip 5pt\noindent
\begin{lem}\label{lemFH} Let $Z\in\bM_n$. Then, there exists a unitary matrix $V\in\bM_n$ such that
$$
{\mathrm{Re\,}}Z \le V|Z|V^*.
$$
\end{lem}
\vskip 5pt\noindent

\begin{proof} Denote by $\Lambda_j$ the set of all $j$-dimensional subspaces of $\bC^n$, and let $S^1$ be the unit sphere in the subspace $S$. Then,
$$
\lambda_j^{\downarrow} ({\mathrm{Re\,}}Z) =\max_{S\in\Lambda_j} \min_{h\in S^1} \langle h, {\mathrm{Re\,}}Zh\rangle= \max_{S\in\Lambda_j} \min_{h\in S^1}  {\mathrm{Re\,}} \langle h, Zh\rangle\le
  \max_{S\in\Lambda_j}\min_{h\in S^1}  \| Zh\|.
$$
We might stop our proof there, or detail a little bit more, hence establishing the min-max principle for singular values from the min-max principle for the eigenvalues:
$$
  \max_{S\in\Lambda_j}\min_{h\in S^1}  \| Zh\|=   \max_{S\in\Lambda_j}\min_{h\in S^1}  \langle h, |Z|^2h\rangle^{1/2}=\bigl(\lambda_j^{\downarrow}(|Z|^2)\bigr)^{1/2}=\lambda_j^{\downarrow}(|Z|).
$$
Combining the inequality in the first part of the proof with the equality of the second part completes the proof.
\end{proof}

Using this lemma, Thompson \cite{Tho76} derived in a simple but clever way his famous inequality:

\vskip 5pt\noindent
\begin{thm}\label{thT} Let $A,B\in\bM_n$. Then there exist unitary matrices $U,V\in\bM_n$ such that 
$$|A+B|\le  U|A|U^*+V|B|V^*.
$$
\end{thm}
\vskip 5pt\noindent
\begin{proof}
	Let $A+B=Q|A+B|$ be a polar decomposition, and extend the polar
	partial isometry $Q$ to a unitary matrix. With $W=Q^*$, we have
	$|A+B|=W(A+B)$. Since the left-hand side is Hermitian,
	\[
	|A+B|=\Ree(WA)+\Ree(WB).
	\]
	Applying the Fan--Hoffman lemma to $WA$ and $WB$, and using
	$|WA|=|A|$ and $|WB|=|B|$, proves the result.
\end{proof}

Taking $A={\mathrm{Re\,}}Z$ and $B=i{\mathrm{Im\,}}Z$ in Theorem~\ref{thT} gives the following Cartesian consequence.

\vskip 5pt\noindent
\begin{cor}\label{corTcartesian} Let $Z\in\bM_n$. Then, there exist unitary matrices $U,V\in\bM_n$ such that 
$$|Z|\le  U|{\mathrm{Re\,}}Z|U^*+V|{\mathrm{Im\,}}Z|V^*.
$$
\end{cor}
\vskip 5pt\noindent

A subadditivity inequality of Aujla and Bourin \cite{AB07}, for positive matrices $A,B\in\bM_n^+$ and nonnegative concave functions $f(t)$ on the positive half-line, states that
\begin{equation}\label{eq:Auj-Bou}
	f(A+B)\le  Uf(A)U^*+Vf(B)V^*.
\end{equation}
Such a function $f(t)$ is necessarily nondecreasing, so using Thompson's inequality we have a refinement of Rotfel'd inequality:

\vskip 5pt\noindent
\begin{thm}\label{thTAB} Let $A,B\in\bM_n$ and let $f:[0,\infty)\to[0,\infty)$ be concave. Then, there exist unitary matrices $U,V$ such that 
$$f(|A+B|)\le  Uf(|A|)U^*+Vf(|B|)V^*.
$$
\end{thm}
\vskip 5pt\noindent

\begin{proof}
By Thompson's inequality, for some unitaries $U_0,V_0$,
$$
|A+B|\le P:=U_0|A|U_0^*+V_0|B|V_0^*.
$$
Since $f$ is nonnegative and concave on $[0,\infty)$, it is nondecreasing. Hence the eigenvalue order gives
$\lambda_j(f(|A+B|))\le \lambda_j(f(P))$ for all $j$. By the unitary orbit order, this means that
$f(|A+B|)\le Wf(P)W^*$ for some unitary $W$. Applying the Aujla--Bourin inequality to $P$ gives
$$
f(P)\le U_1f(U_0|A|U_0^*)U_1^*+V_1f(V_0|B|V_0^*)V_1^*.
$$
Absorbing the unitaries into $U$ and $V$ completes the proof.
\end{proof}

Applying Theorem~\ref{thTAB} to the Cartesian decomposition $Z={\mathrm{Re\,}}Z+i{\mathrm{Im\,}}Z$ yields the next corollary.

\vskip 5pt\noindent
\begin{cor}\label{corABcartesian} Let $Z\in\bM_n$ and let $f:[0,\infty)\to[0,\infty)$ be concave. Then, there exist unitary matrices $U,V\in\bM_n$ such that 
$$f(|Z|)\le  Uf(|{\mathrm{Re\,}}Z|)U^*+Vf(|{\mathrm{Im\,}}Z|)V^*.
$$
\end{cor}
\vskip 5pt\noindent

Note that we always have 
$
\| Z\|_{\infty} \le \| |{\mathrm{Re\,}}Z| + |{\mathrm{Im\,}}Z|\|_{\infty}
$
and so, since $f(t)$ is necessarily nondecreasing,
$$ \| f(|Z|)\|_{\infty} \le \| f(|{\mathrm{Re\,}}Z| + |{\mathrm{Im\,}}Z|)\|_{\infty}.$$
However, for every unitarily invariant norm, or symmetric norm
(to be defined in the next subsection), that is not a scalar multiple of
$\|\cdot\|_{\infty}$, the reverse strict inequality can occur:
\[
\| f(|Z|)\|
>
\| f(|{\mathrm{Re\,}}Z|+|{\mathrm{Im\,}}Z|)\|
\]
for some $Z\in\bM_2$ and some concave function $f$, see \cite{Bou10}.
Thus Corollary~\ref{corABcartesian} is a subtle statement, and
Thompson's inequality highlights a genuinely noncommutative phenomenon.

\section{Symmetric norms}\label{s:symmetric-norms}

The polar decomposition invites us to consider norms on $\bM_n$ such that for all $Z\in\bM_n$,
$$\|Z\|=\| |Z|\|=\| |Z^*| \|.$$ These norms are characterized by  an invariance property, for all unitaries $U,V\in\bM_n$,
$$
\| Z\|= \| UZV\|.
$$
Hence, with the singular value decomposition $Z=U|Z|^{\downarrow}V^*$, we have $\| Z\|=\| |Z|^{\downarrow}\|$, where $|Z|^{\downarrow}:={\mathrm{diag}}(\lambda_1^{\downarrow}(|Z|),\ldots,\lambda_n^{\downarrow}(|Z|) )$. Such a norm is called a symmetric norm on $\bM_n$, or a unitarily invariant one. Furthermore, if $A,B\in\bM_n$ are two contractions, then
$$
\| AZB\| \le \| Z\|.
$$
Indeed, it suffices to consider the case $A,B\in\bM_n^+$ and to note that a positive contraction is a convex combination of Hermitian unitary matrices. In particular, for two positive matrices in $\bM_n$ such that $0\le X\le Y$,
we have $X=Y^{1/2}K^2Y^{1/2}\simeq KYK$ for some positive contraction $K$, so symmetric norms respect the order relation in $\bM_n^+$.

The three simplest examples of symmetric norms on $\bM_n$ are the operator norm $\| \cdot\|_{\infty}$, the Hilbert-Schmidt (or Frobenius)  norm $\|\cdot\|_2$ and the trace norm $\| \cdot\|_1$. For $A\in\bM_n$, its operator norm $\| A\|_{\infty}=\max\{ \|Ax\| \ : \: \|x\|=1\}$, where the symbol $\|\cdot\|$ also stands for the standard norm on $\bC^n$.

\vskip 5pt\noindent
\begin{defin}\label{defHSTN}
The Hilbert-Schmidt norm, or Frobenius norm $\|\cdot\|_2$ on $\bM_n$ is the norm associated to the canonical scalar product $\langle X,Y\rangle ={\mathrm{Tr}}\, X^*Y$. Thus, for all $A\in\bM_n$,
$$
\| A\|_2:= \left\{{\mathrm{Tr}} |A|^2\right\}^{1/2} =\left\{ \sum_{j=1}^n \lambda_j^{\downarrow}(|A|)^2\right\}^{1/2}.
$$
The trace norm $\|\cdot\|_1$ on $\bM_n$ is defined by
$$
\| A\|_1:= {\mathrm{Tr}} |A| = \sum_{j=1}^n \lambda_j^{\downarrow }(|A|).
$$
\end{defin}
\vskip 5pt\noindent

While the triangle inequality is obvious for the operator norm and the Hilbert-Schmidt norm, the case of the trace norm deserves a small proof. With Thompson's inequality at our disposal, the triangle inequality for $\|\cdot\|_1$ becomes obvious. Given $X,Y\in\bM_n$, there exist two unitary matrices $U,V\in\bM_n$  such that:
$$
\| X+Y\|_{1} ={\mathrm{Tr}}\, |X+Y| \le  {\mathrm{Tr}}\, U|X|U^* +  {\mathrm{Tr}}\, V|Y|V^* =\| X\|_1+\|Y\|_{1}.
$$

\vskip 5pt\noindent
\begin{rem} The Cauchy-Schwarz inequality for  $\|\cdot\|_2$ entails that $\|\cdot\|_1$ is a norm. Indeed, given a pair of matrices $X,Y\in\bM_n$ we have a unitary $U$ such that
$$
\| X+Y\|_{1} = {\mathrm{Tr}}\, U(X+Y)
\le |{\mathrm{Tr}}\, UX|+|{\mathrm{Tr}}\,UY|
\le\| X\|_1+\|Y\|_{1}.
$$
Here the last inequality follows from $| {\mathrm{Tr}}\, Z|\le  {\mathrm{Tr}}\, |Z|$ for all $Z\in\bM_n$. This is a consequence of the Cauchy-Schwarz inequality for the decomposition $Z=(V|Z|^{1/2})|Z|^{1/2}$.
\end{rem}
\vskip 5pt\noindent

Another important class of symmetric norms is given in the following definition. These norms form a natural scale between the operator norm and the trace norm.

\vskip 5pt\noindent
\begin{defin}\label{defKFN} The Ky Fan $k$-norms on $\bM_n$, $k=1,\ldots,n$ are the symmetric norms defined by
$$
\| X\|_{(k)}:=\max_{E\in\Pi_k } {\mathrm{Tr}}\, E|X|E =\sum_{j=1}^k\lambda_j^{\downarrow}(|X|)
$$
where $\Pi_k$ denotes the collection of all rank $k$ projections in $\bM_n$.
\end{defin}
\vskip 5pt\noindent

Hence $\|\cdot\|_1=\|\cdot\|_{(n)}$ and $\|\cdot\|_{\infty}=\|\cdot\|_{(1)}$, where  $\|\cdot\|_{\infty}$ stands for the operator norm.  Let us show that Definition \ref{defKFN} is consistent:

\begin{proof}
The second equality follows from the min--max principle, for any $k$-dimensional subspace $\trS\subset \bC^n$ (here we take $\trS$ as the range of $E$) we have  $\lambda_j^{\downarrow}(|X|_{\trS})\le \lambda_j^{\downarrow}(|X|)$. It is clear  that $\| X\|_{(k)}$ is  a homogeneous unitarily invariant functional. To prove that this is a symmetric norm, we use Thompson's inequality to obtain the triangle inequality for $\|\cdot\|_{(k)}$ as follows: For some rank $k$ projection $E$,
$$
\| X+Y\|_{(k)}=  {\mathrm{Tr}}\, E|X+Y|E, 
$$
and Thompson's inequality yields two unitary matrices $U$ and $V$ such that
\begin{align*}
\| X+Y\|_{(k)} &\le  {\mathrm{Tr}}\, EU|X|U^*E +  {\mathrm{Tr}}\, EV|Y|V^*E \\
&= {\mathrm{Tr}}\, U^*EU|X|U^*EU +  {\mathrm{Tr}}\, V^*EV|Y|V^*EV \\
&\le \| X\|_{(k)} + \| Y\|_{(k)}.
\end{align*}
So the Ky Fan functionals are symmetric norms.
\end{proof}

The following two propositions are other ways to look at Ky Fan $k$-norms. These propositions also show that the 
 Ky Fan $k$-norms are indeed norms.

\vskip 5pt\noindent
\begin{prop}\label{KFN2} The Ky Fan $k$-norms on $\bM_n$,  $k=1,\ldots,n$ satisfy
$$
\| X\|_{(k)}=\max_{A,B\in\Gamma_k } |{\mathrm{Tr}}\, AXB|
$$
where $\Gamma_k$ denotes all rank $\le k$ contractions in $\bM_n$. Equivalently,
$$
\| X\|_{(k)}=\max_{C\in\Gamma_k } |{\mathrm{Tr}}\, CX|.
$$
\end{prop}
\vskip 5pt\noindent

\begin{proof} 
A form of the polar decomposition of $X$ shows  that, thanks to Cauchy-Schwarz,
\begin{align*}
 |{\mathrm{Tr}}\, AXB|^2 &= \left|{\mathrm{Tr}}\, A|X^*|^{1/2}U|X|^{1/2}B\right|^2 \\
&\le  {\mathrm{Tr}}\, |X^*|^{1/2}A^*A|X^*|^{1/2}\,
 {\mathrm{Tr}}\, |X|^{1/2} BB^*|X|^{1/2}\\
&\le {\mathrm{Tr}}\, |X^*|^{1/2}E|X^*|^{1/2}\,
 {\mathrm{Tr}}\, |X|^{1/2} F|X|^{1/2},
\end{align*}
where $E$ is the support projection of $A^*A$ and $F$ is the support projection of $BB^*$.
From the min--max principle it follows that, for $j\le k\le n$ and $Z\in\bM_n$, $\lambda_j^{\downarrow}(E|Z|E)\le  \lambda_j^{\downarrow}(|Z|)$ since 
$E|Z|E$ corresponds to a compression of $|Z|$ on a $k$-dimensional subspace. Hence
$$
 |{\mathrm{Tr}}\, AXB|\le \sum_{j=1}^k\lambda_j^{\downarrow}(|X|).$$
If $B=P$ is  a projection corresponding to the $k$ largest eigenvalues of $|X|$, and $A=PU^*$, we have equality in the first maximum stated in the proposition, 
$$
 |{\mathrm{Tr}}\, AXB| = {\mathrm{Tr}}\, AXB=  {\mathrm{Tr}}\, PU^*XP = \sum_{j=1}^k\lambda_j^{\downarrow}(|X|).$$
So equality holds in the second maximum with $C=PU^*$.
Thus Proposition \ref{KFN2} is equivalent to the initial Definition \ref{defKFN} of the $k$-norms. 
\end{proof}

The proof shows that it suffices to consider partial isometries in $\Gamma_k$. So we can state the next proposition, where
 $I_k$ denotes the identity of $\bM_k$ and  ${\mathrm{Tr}}_k$
the trace on $\bM_k$.

\vskip 5pt\noindent
\begin{prop}\label{KFN3}  The Ky Fan $k$-norms on $\bM_n$,  $k=1,\ldots,n$ satisfy
$$
\| X\|_{(k)}:= \max_{V,W\in\Theta(k,n)}\left| {\mathrm{Tr}}_k\, V^*XW\right|.
$$
where $\Theta(k,n)$ stands for the set of $n\times k$ isometry matrices $J$, i.e., $J^*J=I_k$.
\end{prop}
\vskip 5pt\noindent

\begin{proof} Note that with (${\mathrm{Tr}}={\mathrm{Tr}}_n $) we have  
$$
  {\mathrm{Tr}}_k V^*XW =  {\mathrm{Tr}}\, V_0^*XW_0
$$
where $W_0,V_0\in\bM_n$ are obtained from $W,V\in \bM_{n,k}$ by adding $n-k$ zero columns. Hence
$$
|  {\mathrm{Tr}}_k V^*XW|\le \| X\|_{(k)}.
$$
On the other hand, observe that
$$
{\mathrm{Tr}}_k V^*XW ={\mathrm{Tr}}_n XWV^*\, (= {\mathrm{Tr}} XWV^*)
$$
and choose $W\in\bM_{n,k}$ any isometry with range projection $P=WW^*$, a spectral projection of $|X|$ corresponding to the $k$ largest eigenvalues. Then pick $V=UW$, where $U$ still holds for the unitary factor of $X$. We then have
$$
{\mathrm{Tr}}_k V^*XW = {\mathrm{Tr}} XWW^*U^*=  {\mathrm{Tr}} XPU^*= \sum_{j=1}^k\lambda_j^{\downarrow}(|X|)$$
and the proof is complete.
\end{proof}

Since  $\|\cdot\|_2$ is a very nice norm, we may expect special results for it.
Recall that  $\bM_n^{\sa}$ is the Hermitian part of $\bM_n$ and $\bM_n^+$ is the positive semidefinite cone.

\vskip 5pt\noindent
\begin{defin} A pair $(A,B)$ in $\bM_n^{\sa}$ is  {\sl a monotone pair} if
$A=f(C)$ and $B=g(C)$ for two non-decreasing functions $f(t),\,g(t)$ and some $C\in\bM_n^{\sa}$. If $f(t)$ is non-decreasing and
$g(t)$ is non-increasing, then  $(A,B)$ is {\sl antimonotone}.
\end{defin}
\vskip 5pt\noindent

This notion leads to the following Hilbert--Schmidt norm comparison.

\vskip 5pt\noindent
\begin{thm}\label{thmonot1} Let $Z\in\bM_n$ be a normal matrix. If $(A,B)$ is a monotone pair in $\bM_n^+$ then, 
$$
\| AZB \|_2 \le \| ZAB \|_2.
$$
If $(A,B)$ is antimonotone, then the inequality reverses.
\end{thm}
\vskip 5pt\noindent

Since $(A,B)$ is a monotone (antimonotone) pair in $\bM_n^+$ if and only if so is $(A^{1/2}, B^{1/2})$, squaring both sides, the theorem is equivalent to the trace inequalities in the next corollary. 

\vskip 5pt\noindent
\begin{cor}\label{cormonot} Let $Z\in\bM_n$ be a normal matrix. If $(A,B)$ is a monotone pair in $\bM_n^{\sa}$ then, 
$$
{\mathrm{Tr}\,} Z^*AZB \le {\mathrm{Tr}\,} Z^*ZAB. 
$$
If $(A,B)$ is antimonotone, then the inequality reverses.
\end{cor}
\vskip 5pt\noindent

\begin{proof}[Proof of Corollary~\ref{cormonot}] Replacing $A,B$ by $A+kI, B+kI$ for $k$ large enough, we may assume $A,B\in\bM_n^+$. It suffices to prove the monotone case, the antimonotone case follows by replacing $B$ with $cI-B$ for $c>0$ large enough. By layer decomposition of $C$, hence of $A$ and $B$, it suffices to show that 
$$
{\mathrm{Tr}\,} Z^*EZF \le {\mathrm{Tr}\,} Z^*ZEF
$$
for every pair of projections $E$ and $F$ such that $E\le F$ or $F\le E$. 

Suppose $E\le F$. Then, since $Z^*Z=ZZ^*$,
 $${\mathrm{Tr}\,} Z^*EZF\le {\mathrm{Tr}\,} Z^*EZ={\mathrm{Tr}\,} ZZ^*E= {\mathrm{Tr}\,} Z^*ZEF.$$
Suppose $F\le E$. Then,
${\mathrm{Tr}\,} Z^*EZF\le {\mathrm{Tr}\,} Z^*ZF= {\mathrm{Tr}\,} Z^*ZEF.$
\end{proof}

 Letting $Z=|h\rangle\langle h|$ be a rank one projection, Corollary \ref{cormonot} says that, for a monotone pair $(A,B)$ in $\bM_n$,
$$
\langle h, Ah\rangle \langle h, Bh\rangle \le \langle h, ABh\rangle.
$$
Computing the scalar products in a basis diagonalizing $A,B$ we see that we obtain 
Chebyshev's inequality for pairs of non-decreasing continuous functions on $(0,1)$, 
$$
\int_0^1 f(t)\, {\mathrm{d}}t\ \int_0^1  g(t) \,{\mathrm{d}}t\le \int_0^1  f(t)g(t) \, {\mathrm{d}}t.
$$
Thus Theorem \ref{thmonot1} is a matrix Chebyshev inequality. Another interesting special case appears by
letting $Z$ be a unitary matrix:  We may then obtain von Neumann trace inequalities \cite{Neu37}: 

\vskip 5pt\noindent
\begin{lem}\label{lemVNRI} Let $X,Y\in\bM_n$. Then 
$$
\left|{\mathrm{Tr}\,} XY \right| \le \sum_{j=1}^n \lambda_j^{\downarrow}(|X|) \lambda_j^{\downarrow}(|Y|) 
$$
\end{lem}
\vskip 5pt\noindent

\begin{proof}
First
if $X$ and $Y$ are both positive then, for some unitary $Z$,
$ZXZ^*$ and $Y$ form a monotone pair, and Corollary \ref{cormonot} yields
$$
{\mathrm{Tr}\,} XY ={\mathrm{Tr}\,}(ZXZ^*)ZYZ^*\le {\mathrm{Tr}\,} ZXZ^*Y=\sum_{j=1}^n \lambda_j^{\downarrow}(|X|) \lambda_j^{\downarrow}(|Y|).
$$
For  general $X$ and $Y$,
 we consider  polar decompositions $X=U|X|$, $Y=V|Y|$ and we note that
\begin{align*}
\left| {\rm Tr}\, XY\right| &= \left| {\rm Tr}\, |X^*|^{1/2}U|X|^{1/2}V|Y|\right| \\
&= \left| {\rm Tr}\, |Y|^{1/2}|X^*|^{1/2}U|X|^{1/2}V|Y|^{1/2}\right| \\
&\le \{{\rm Tr}\, |Y||X^*|\}^{1/2} \{{\rm Tr}\, |X|V|Y|V^*\}^{1/2}
\end{align*}
by using the Cauchy-Schwarz inequality $\left| {\rm Tr}\,
A^*B\right|\le \{{\rm Tr}\, A^*A\}^{1/2} \{{\rm Tr}\,
B^*B\}^{1/2}$. Hence the general case follows
from the positive one.
\end{proof}

The Hermitian version follows from the same rearrangement argument.

\vskip 5pt\noindent
\begin{lem}\label{lemVN2} Let $A,B\in\bM_n^{\sa}$. Then 
$$
{\mathrm{Tr}\,} AB \le \sum_{j=1}^n \lambda_j^{\downarrow}(A) \lambda_j^{\downarrow}(B).
$$
\end{lem}
\vskip 5pt\noindent

\begin{proof} Exactly the same proof as the previous one in the case of $X,Y\in\bM_n^+$.
\end{proof}

The operator norm is denoted by $\|\cdot\|_{\infty}$ as it is the limit of Schatten $p$-norms, $p\to\infty$, given in the next definition.

\vskip 5pt\noindent
\begin{defin}\label{defSpN} The Schatten $p$-norms on $\bM_n$, $1\le p<\infty$, are the symmetric norms defined by
$$
\| A\|_p:= \left\{{\mathrm{Tr}} |A|^p\right\}^{1/p} =\left\{ \sum_{j=1}^n \lambda_j^{\downarrow}(|A|)^p\right\}^{1/p}.
$$
\end{defin}
\vskip 5pt\noindent

We can prove that the Schatten $p$-norms are indeed symmetric norms in two different ways. First we may note the matrix Hölder inequality:

\vskip 5pt\noindent
\begin{prop} Let $X,Y\in \bM_n$ and let $1\le p,q\le\infty$ be conjugate, i.e., $p^{-1}+q^{-1}=1$. Then
$$
|{\mathrm{Tr}} X^*Y| \le \| X\|_p\|Y\|_q
$$
\end{prop}
\vskip 5pt\noindent

The matrix Hölder inequality follows from the classical one by using Lemma \ref{lemVNRI}. As in the classical case we then infer that $\|\cdot\|_p$ and $\|\cdot\|_q$ are dual norms. Thus we have the matrix Minkowski inequality:

\vskip 5pt\noindent
\begin{prop}\label{Sptriang} Let $A,B\in \bM_n$ and  $1\le p\le\infty$. Then
$$
 \| A+B\|_p\le \| A\|_p +\|B\|_p.
$$
\end{prop}
\vskip 5pt\noindent

We can give a proof of this triangle inequality by using Thompson's inequality. It is our second proof that the Schatten $p$-norms are indeed norms.

\begin{proof} The case $p=\infty$ is well-known, so let $p<\infty.$ Suppose first that $\|A\|_p=\|B\|_p=1$ and let
$0\le t\le 1$.
Then, there exist two unitaries $U,V$ such that
$$
\| tA+(1-t)B\|_p^p = {\mathrm{Tr}} |tA+(1-t)B|^p \le {\mathrm{Tr}} (tU|A|U^*+(1-t)V|B|V^*)^p.
$$
so that, using the convexity inequality \eqref{vNconv},
$$
\| tA+(1-t)B\|_p^p \le  {\mathrm{Tr}} (tU|A|^pU^*+(1-t)V|B|^pV^*)=1
$$
and so $\| tA+(1-t)B\|_p\le 1$.
Now for the general case, $A,B\neq 0$, apply the previous case with
$$
\frac{1}{\| A\|_p}A, \quad \frac{1}{\| B\|_p}B, \quad t=\frac{\|A\|_p}{\|A\|_p+\|B\|_p}.
$$
This yields $\| A+B\|_p \le  \| A\|_p +\|B\|_p$. 
\end{proof}

\vskip 5pt\noindent
\begin{rem}\label{originalproof} A third proof of Proposition \ref{Sptriang} is possible by using only the convexity inequality \eqref{vNconv}. To do so, consider the convex function $t\mapsto|t|^p$ and the Hermitian operators
$$
A_0=\begin{pmatrix} 0&A \\
A^*&0
\end{pmatrix}, \qquad
B_0=\begin{pmatrix} 0&B \\
B^*&0
\end{pmatrix}.
$$
Arguing as in the second part of the previous proof with $A_0$, $B_0$ and $t\mapsto|t|^p$ yields Proposition \ref{Sptriang}.
\end{rem}
\vskip 5pt\noindent

 Proposition \ref{Sptriang} is a special case of the next proposition (with the trace norm).

\vskip 5pt\noindent
\begin{prop}\label{normpconv} Let $\|\cdot\|$ be a symmetric norm on $\bM_n$ and  let $p>1$. Then, 
$$
X\mapsto \| |X|^p\|^{1/p}
$$
is a symmetric norm on $\bM_n$.
\end{prop}
\vskip 5pt\noindent

\begin{proof} The proof is similar to the previous one. Alternatively, let $\Phi$ be the symmetric gauge function associated with $\|\cdot\|$. Thus
$\|X\|=\Phi(s(X))$, where $s(X)$ is the vector of singular values of $X$ in nonincreasing order. Thompson's inequality, or equivalently the Ky Fan inequalities, gives
$$
s(A+B)\prec_w s(A)+s(B).
$$
For nonnegative vectors $x,y$, define $\Psi(x)=\Phi(x^p)^{1/p}$. We first note that $\Psi$ is a norm. Indeed, using the dual gauge function $\Phi^*$ and the scalar Minkowski inequality,
\begin{align*}
\Phi((x+y)^p)^{1/p}
&=\sup_{\Phi^*(z)\le 1,\ z\ge0}\left(\sum_i z_i(x_i+y_i)^p\right)^{1/p}\\
&\le \Phi(x^p)^{1/p}+\Phi(y^p)^{1/p}.
\end{align*}
Moreover, if $0\le x\prec_w y$, then $x^p\prec_w y^p$ for $p>1$, and hence $\Phi(x^p)\le\Phi(y^p)$. Therefore,
$$
\bigl\| |A+B|^p\bigr\|^{1/p}
=\Psi(s(A+B))
\le \Psi(s(A)+s(B))
\le \Psi(s(A))+\Psi(s(B)),
$$
which is the required triangle inequality. The other norm properties and the unitary invariance are immediate.
\end{proof}

If $A\in\bM_{n_1+n_2}$ is partitioned as
$$
A=\begin{bmatrix} A_1&X \\ Y&A_2\end{bmatrix}
$$
then 
$$
\begin{bmatrix} A_1&0 \\ 0&A_2\end{bmatrix}
=\frac{1}{2}\left\{\begin{bmatrix} A_1&X \\ Y&A_2\end{bmatrix} +\begin{bmatrix}I&0 \\ 0&-I\end{bmatrix}\begin{bmatrix} A_1&X \\ Y&A_2\end{bmatrix}\begin{bmatrix}I&0 \\ 0&-I\end{bmatrix}\right\}.
$$
Hence, arguing by induction or recursively, we have the basic fact:

\vskip 5pt\noindent
\begin{prop}\label{pinchtrace}  If $\Delta(A)$ is the diagonal part of a  matrix $A\in\bM_n$, then, for all symmetric norms, 
$$
\| \Delta(A) \| \le \| A\|.
$$
\end{prop}
\vskip 5pt\noindent

We close this section with a historical comment and a related conjecture. Monotone pairs were studied by the first author; see \cite{Bou06} for von Neumann trace rearrangement inequalities and further references.

\vskip 5pt\noindent
\begin{rem} If $(A,B)$ is a monotone pair in $\bM_n^+$, it is natural to conjecture  \cite{Bou99} that, for every symmetric norm,
$$
\| AZB \| \le \sqrt{2}\| ZAB \|.
$$
The remaining material of this section is rather classical, but our proofs are original. In particular, the proof of the Minkowski inequality in
Proposition~\ref{Sptriang}, mentioned in Remark~\ref{originalproof},
is a new and simple application of the trace Jensen inequality \eqref{vNconv}.
\end{rem}
\vskip 5pt\noindent

\section{Majorization and unitary orbits}\label{s:majorization}

We now recall the majorization language needed to pass from operator inequalities to all symmetric norms. The emphasis is on convex hulls and unitary orbits.

 For a pair of Hermitian matrices $A,B\in\bM_n$, the weak majorization, or submajorization, relation $A\prec_w B$ means that
 \begin{equation}\label{submajdef}
\sum_{j=1}^k\lambda_j^{\downarrow}(A) \le  \sum_{j=1}^k\lambda_j^{\downarrow}(B)
\end{equation}
for all $k=1,\ldots, n$. If,
 furthermore,   equality holds for $k=n$, that  is, if $A$ and $B$ have  the same trace, then we say that $A$  is majorized by $B$, and we write 
\begin{equation}\label{majorization}
A\prec B.
\end{equation}
The notation in \eqref{majorization} will be used throughout this section. 
For each $k$, the sum of the $k$ largest eigenvalues occurring in
\eqref{submajdef},
\begin{equation}\label{sumkeig}
	A\mapsto \sum_{j=1}^k\lambda_j^{\downarrow}(A),
\end{equation}
is a subadditive functional on $\bM_n^{\sa}$. This follows from the
Ky Fan variational principle
\begin{equation}\label{variassumeig}
	\sum_{j=1}^k\lambda_j^{\downarrow}(A)
	=
	\max_{\substack{P=P^*=P^2\\ \Tr P=k}}\Tr(PA).
\end{equation} 
This is a direct consequence of the min-max principle. On the positive cone $\bM_n^+$, the subadditivity of \eqref{sumkeig}
is exactly the subadditivity of the Ky Fan $k$-norms.

Majorization and weak majorization relations lie at the heart of matrix analysis; see, for example, Ando's survey \cite{And89}. A simple example is Schur's theorem, a direct consequence of \eqref{variassumeig}:

\vskip 5pt\noindent
\begin{prop}\label{Schur} If $\Delta(A)$ is the diagonal part of a Hermitian matrix $A\in\bM_n$, then 
$$
\Delta(A)\prec A.
$$
\end{prop}
\vskip 5pt\noindent

Despite its simplicity, this result has been extremely influential in the development of matrix analysis.
Consideration of Hermitian or positive two-by-two matrices suggests that majorization and weak majorization are related to convexity properties. To state these convexity results,  we need a few notations.

{\bf Notations. } If $R\in\bM_n$ is a diagonal matrix, $R={\mathrm{diag}}(r_1,\ldots,r_n)$  and $\sigma\in\mathfrak{S}_n$ is a permutation of $\{1,\ldots,n\}$, we define  the eigenvalue vectors:
$$
r=(r_1,\ldots,r_n), \qquad
r_{\sigma}:=(r_{\sigma(1)},\ldots,r_{\sigma(n)})
$$
In particular, If $A\in\bM_n$ is Hermitian, we define  the eigenvalue vectors:
 $$a^{\downarrow}=(a_1^{\downarrow},\ldots, a_n^{\downarrow}):=(\lambda_1^{\downarrow}(A),\ldots, \lambda_n^{\downarrow}(A))$$
and
$$
a^{\downarrow}_{\sigma}=(a_{\sigma(1)}^{\downarrow},\ldots,a_{\sigma(n)}^{\downarrow}):=(\lambda_{\sigma(1)}^{\downarrow}(A),\ldots, \lambda_{\sigma(n)}^{\downarrow}(A)).
$$
Further we denote by $\trE_n$ the set of vectors of $\bR^n$ with $0$-$1$ entries, thus $\eps\in\trE_n$ if and only if $\eps=(\eps_1,\ldots, \eps_n)$ where $\eps_i\in\{0,1\}$. For such a vector  $\eps$ and an eigenvalue vector $c_{\sigma}$ associated to a diagonal matrix $C\in\bM_n$ and a permutation $\sigma$, we define the product 
$$
\eps c_{\sigma}:=(\eps_1c_{\sigma(1)},\ldots,\eps_nc_{\sigma(n)}).
$$

We may then state a fundamental proposition:  majorization is equivalent to the property that an eigenvalue vector belongs to the convex hull of the permutations of another eigenvalue vector.

\vskip 5pt\noindent
\begin{prop}\label{majconv}
Let $A,B$ be two Hermitian matrices in $\bM_n$. Then,
$$
A\prec B \iff a^{\downarrow} \in {\mathrm{conv}}\{\, b_{\sigma}^{\downarrow} \ : \ \sigma\in\mathfrak{S}_n\,\}.
$$
\end{prop}
\vskip 5pt\noindent

\begin{proof}  Assume first that $A\prec B$. If $ a^{\downarrow} \notin {\mathrm{conv}}\{\,  b_{\sigma}^{\downarrow} \ : \  \sigma\in\mathfrak{S}_n\,\}$, then we would have a vector
 $x=(x_1,\ldots, x_n)$ such that, for any permutation $\sigma$,
$$
\sum_{i=1}^n x_ia_i^{\downarrow} > \sum_{i=1}^n x_ib_{\sigma(i)}^{\downarrow}.
$$
and so, with a suitable choice of $\sigma$,
$$
\sum_{i=1}^n x_ia_i^{\downarrow} > \sum_{i=1}^n x_i^{\downarrow}b_i^{\downarrow}
$$
where $x^{\downarrow}$ denotes the vector obtained from $x$ by nonincreasing rearrangement of its components. By the commutative case of von Neumann trace inequality (Lemma \ref{lemVN2}), this would entail:
\begin{equation}\label{proofmaj}
\sum_{i=1}^n x_i^{\downarrow}a_i^{\downarrow} > \sum_{i=1}^n x_i^{\downarrow}b_i^{\downarrow}.
\end{equation}
Since $\sum_{i=1}^n a_i^{\downarrow} =\sum_{i=1}^n b_i^{\downarrow}$, we would also have
$$
\sum_{i=1}^n y_i^{\downarrow}a_i^{\downarrow} > \sum_{i=1}^n y_i^{\downarrow}b_i^{\downarrow}.
$$
where $y_i=x_i+c$, the constant $c>0$ being large enough to ensure $y^{\downarrow}_n>0$. Hence we may assume that $x_n^{\downarrow}>0$. 
 Observe that \eqref{proofmaj}  implies
\begin{align*}
x_n^{\downarrow}\left\{\sum_{i=1}^n (a_i^{\downarrow} -b_i^{\downarrow})\right\} 
 +(x^{\downarrow}_{n-1}-x_n^{\downarrow}&)
\left\{\sum_{i=1}^{n-1} (a_i^{\downarrow} -b_i^{\downarrow})\right\}+\cdots \\
 &+\cdots +(x_1^{\downarrow}-x_2^{\downarrow})\left\{a_1^{\downarrow} -b_1^{\downarrow}\right\}>0.
\end{align*}
As $x_n^{\downarrow}>0$,  this  contradicts $A\prec B$. Therefore
 $ a^{\downarrow} \in {\mathrm{conv}}\{\, b_{\sigma}^{\downarrow} \ :  \ \sigma\in\mathfrak{S}_n\,\}$.

Conversely, suppose that
$$
a^{\downarrow}=\sum_{r=1}^N t_r b_{\sigma_r}^{\downarrow},
\qquad t_r\ge0,\quad \sum_{r=1}^N t_r=1.
$$
For each $k$, the sum of the $k$ largest coordinates is a convex function. Hence
$$
\sum_{j=1}^k a_j^{\downarrow}
\le \sum_{r=1}^N t_r\sum_{j=1}^k b_j^{\downarrow}
=\sum_{j=1}^k b_j^{\downarrow}.
$$
For $k=n$, equality holds. Thus $A\prec B$.
\end{proof}

The convex-hull characterization immediately yields a unitary-orbit formulation of weak majorization.

\vskip 5pt\noindent
\begin{cor}\label{wmajorbit}
Let $A,B$ be two Hermitian matrices in $\bM_n$. If $A\prec_w B$ then there exist unitary matrices $\{U_i\}_{i=1}^{n+1}$ such that
$$
A\le \sum_{i=1}^{n+1} t_{i} U_iBU_i^*
$$
for a convex combination with weights $0\le t_i\le 1$, $\sum_{i=1}^{n+1}t_i=1$.
\end{cor}
\vskip 5pt\noindent

\begin{proof} Consider the spectral decomposition
$
B=\sum_{k=1}^n\lambda_k^{\downarrow}(B)E_k,
$
and define
$$
B':= \left(\sum_{k=1}^{n-1}\lambda_k^{\downarrow}(B)E_k\right) + \mu E_n
$$
where $\mu:= \lambda_n^{\downarrow}(B)-{\mathrm{Tr\,}} B +{\mathrm{Tr\,}} A$  is chosen so that
$$
A\prec B'.
$$
The weak-majorization inequality for $k=n$ gives $\mu\le\lambda_n^{\downarrow}(B)$, and hence $B'\le B$.
By Proposition~\ref{majconv} and Carathéodory's theorem in \(\mathbb R^n\),
\(A\) can be written as a convex combination of at most \(n+1\)
elements in the unitary orbit of \(B'\):
\[
A= \sum_{i=1}^{n+1} t_i U_iB'U_i^*,
\]
where some coefficients may be zero. This completes the proof.
\end{proof}

Taking diagonal parts in the preceding orbit inequality gives a useful refinement.

\vskip 5pt\noindent
\begin{cor}\label{wmajdiag}
Let $A,B$ be two Hermitian matrices in $\bM_n$. If $A\prec_w B$ then there exists a diagonal matrix $\Delta\in\bM_n^+$ such that
$$
A^{\downarrow} +\Delta\prec B^{\downarrow}
$$
\end{cor}
\vskip 5pt\noindent

\begin{proof} By the inequality in the previous corollary, 
$$
A^{\downarrow}  \le V\left\{\sum_{i=1}^{n+1} t_{i} U_iBU_i^*\right\} V^*
$$
for some unitary $V$. Letting
$$
C= V\left\{\sum_{i=1}^{n+1} t_{i} U_iBU_i^*\right\} V^* -A^{\downarrow}
$$
we have $C\ge 0$. The matrix $A^{\downarrow}+C$ is a convex combination of unitary conjugates of $B$, and so
$$
A^{\downarrow} +C\prec B^{\downarrow}.
$$
By Schur's theorem,
$$
A^{\downarrow}+\Delta(C)=\Delta(A^{\downarrow}+C)\prec A^{\downarrow}+C.
$$
Therefore, with $\Delta:=\Delta(C)$ and by transitivity, we have $A^{\downarrow} +\Delta \prec B^{\downarrow}$.
\end{proof}

Combining this diagonal description with the monotonicity of symmetric norms gives Ky Fan's dominance principle.

\vskip 5pt\noindent
\begin{cor}\label{KFprinciple}
Let $A,B\in \bM_n$. Then,
$$
|A|\prec_w |B| \iff \| A\| \le \| B\| \ for \ all \ symmetric \ norms.
$$ 
\end{cor}
\vskip 5pt\noindent

\begin{proof} The implication $\Rightarrow$ is a direct consequence of Corollary \ref{wmajorbit}, the other implication is trivial.
\end{proof}

For positive matrices, the same criterion has a purely eigenvalue-vector formulation.

\vskip 5pt\noindent
\begin{cor}\label{wmajconv}
Let $A,B$ be two positive matrices in $\bM_n$. Then,
$$
A\prec_w B \iff a^{\downarrow} \in {\mathrm{conv}}\{\, \eps b_{\sigma}^{\downarrow} \ : \ \eps\in\trE_n, \ \sigma\in\mathfrak{S}_n\,\}.
$$
\end{cor}
\vskip 5pt\noindent

\begin{proof} By Corollary \ref{wmajdiag}
there exists a diagonal matrix $\Delta\in\bM_n^+$ such that
$$
A^{\downarrow} +\Delta\prec B^{\downarrow}.
$$ 
By Proposition \ref{majconv},
$$
a^{\downarrow}+\delta \in {\mathrm{conv}}\{\, b_{\sigma}^{\downarrow} \ : \ \sigma\in\mathfrak{S}_n\,\}.
$$
Note that $a^{\downarrow}\in{\mathrm{conv}}\{\, \eps(a^{\downarrow}+\delta) : \ \eps\in\trE_n\,\}$ and so for each $\nu\in\trE_n$, we have
$$
\nu(a^{\downarrow}+\delta)\in{\mathrm{conv}}\{\, \nu b_{\sigma}^{\downarrow} \ : \ \sigma\in\mathfrak{S}_n\,\}.
$$
Therefore
$
a^{\downarrow} \in {\mathrm{conv}}\{\, \eps b_{\sigma}^{\downarrow} \ : \ \eps\in\trE_n, \ \sigma\in\mathfrak{S}_n\,\}.
$
Conversely, each vector $\eps b_{\sigma}^{\downarrow}$ is weakly majorized by $b^{\downarrow}$, and weak majorization is preserved under convex combinations. Hence the convex-hull condition implies $A\prec_w B$.
\end{proof}

For classical proofs of Ky Fan's principle, see \cite{Bha97,HP14,HJ91}; for the unitary-orbit viewpoint used here, see \cite{BL12,BL26a}.

\vskip 5pt\noindent
\begin{rem} Corollary \ref{KFprinciple} is Ky Fan's principle, a fundamental fact. The statement of Corollary \ref{wmajdiag} is somewhat original, and its proof is new. It is based on a systematic approach to majorization through inequalities related to unitary orbits. Stronger results than Corollary \ref{wmajorbit}, together with several other unitary-orbit results, will be introduced in Section~\ref{s:hybrid}.
\end{rem}
\vskip 5pt\noindent

\section{The operator right modulus and  triangle inequalities} \label{s:right}

In this section, we survey some triangle-type estimates for the operator right modulus. In contrast with Thompson's inequality, where unitary conjugations are allowed on the right-hand side, we compare
$
\left|\sum_kA_k\right|
$
directly with
$
\sum_k\left|A_k\right|.
$
Such direct comparisons are much more delicate: in general they fail without additional assumptions, suitable constants, or averaging over unitary orbits. We begin with elementary examples that explain these obstructions, and then review several positive results under natural hypotheses.

\subsection{Operator version} 
As we know, the operator inequality
\[
|A+B|\le |A|+|B|, \qquad A,B\in \mathbb{M}_n,
\]
does not hold in general; for example, consider
$
A=
\begin{pmatrix}
	1&1\\
	0&0
\end{pmatrix},
$ and $
B=
\begin{pmatrix}
	1&-1\\
	0&0
\end{pmatrix}.
$
This elementary example explains why the unitary conjugations in
Thompson's triangle inequality are indispensable.

Naturally, one may ask whether there exists a finite constant \(C(n)\) such that
\[
|A+B|
\le
C(n)\bigl(|A|+|B|\bigr),
\qquad A,B\in\mathbb{M}_n.
\]
Unfortunately, no such finite constant exists when \(n\ge 2\); equivalently,
$
C(n)=\infty.
$
Indeed, it suffices to consider the matrices
\[
A_t=
\begin{pmatrix}
	1&t\\
	0&0
\end{pmatrix},
\qquad
B_t=
\begin{pmatrix}
	1&-t\\
	0&0
\end{pmatrix},
\qquad t>0.
\]
A direct computation gives
\[
|A_t+B_t|=
\begin{pmatrix}2&0\\0&0\end{pmatrix},
\qquad
|A_t|+|B_t|=\frac{2}{\sqrt{1+t^2}}
\begin{pmatrix}1&0\\0&t^2\end{pmatrix}.
\]
Thus any such constant must satisfy $C(n)\ge\sqrt{1+t^2}$ for all $t>0$.

If the operators are invertible, however, the situation changes substantially.
Before presenting the corresponding results, we recall that the condition number
of an invertible matrix \(A\) is defined by
\[
\kappa(A):=\|A\|_\infty\|A^{-1}\|_\infty .
\]

Lee \cite{Lee12} proved the following condition-number-dependent triangle
inequality, in which the constant depends only on a common upper bound
\(\omega\) for the condition numbers of the matrices. Zhang \cite{Zha26a}
later showed, by an explicit example, that this dependence on \(\omega\) is
sharp.
 \vskip 5pt\noindent
\begin{thm}[{\cite{Lee12,Zha26a}}]\label{thm:Lee-triangle-condition-number}
 	Let \(A_1,\dots,A_m\in \mathbb M_n\) be invertible matrices whose condition numbers are bounded by \(\omega\ge1\). Then
 	\begin{equation}\label{eq:Lee}
 		\absop{A_1+\cdots+A_m}
 		\le
 		\frac{\omega+1}{2\sqrt{\omega}}\,
 		\bigl(\absop{A_1}+\cdots+\absop{A_m}\bigr),
 	\end{equation}
 	where the constant is sharp as a function of $\omega$.
 \end{thm}
\vskip 5pt\noindent
 
Moreover, Zhang \cite{Zha26a} obtained the following more refined version of \eqref{eq:Lee} in terms of the condition numbers.
\vskip 5pt\noindent
\begin{thm}[\cite{Zha26a}]\label{thm:Zhang-triangle-condition-number}
	Let \(A_1,\dots,A_m\in \mathbb M_n\) be invertible matrices whose condition numbers are $\kappa_1,\ldots,\kappa_m$.
	Then
	\[
	\bigl|A_1+\cdots+A_m\bigr|
	\le
	\frac{\kappa_1+1}{2\sqrt{\kappa_1}}\,
	|A_1|+\cdots+ 	\frac{\kappa_m+1}{2\sqrt{\kappa_m}}\,|A_m|,
	\]
	where the coefficient function
	$
	c(\kappa):=\frac{\kappa+1}{2\sqrt{\kappa}}
	$
	is individually sharp for each term: for every \(\kappa\ge 1\), the coefficient \(c(\kappa)\)
	in front of a term of condition number \(\kappa\) cannot be replaced by any smaller constant,
	even if all the other coefficients are kept unchanged.
\end{thm}
\vskip 5pt\noindent

If one allows a single unitary conjugation on the right-hand side, Bourin and
Lee \cite{BL26e} obtained the following direct comparison between $|A+B|$
and $|A|+|B|$ for normal matrices.

\vskip 5pt\noindent
\begin{prop}[\cite{BL26e}]\label{cornorm} Let $A,B\in\bM_n$ be normal. Then, for some unitary $V\in\bM_n$,
	$$
	|A+ B| \le |A|+|B| +\frac{1}{4} V(|A|+|B|)V^*.
	$$
	If $n\ge 3$, the constant $1/4$ is the smallest possible one.
\end{prop}
\vskip 5pt\noindent

It is natural to ask what remains true if the normality assumption in
Proposition~\ref{cornorm} is removed and arbitrary complex matrices are
allowed. More precisely, we pose the following problem.

\begin{prob}
	Let \(n\ge 2\). Determine all pairs
$
	\bigl(C_1(n),C_2(n)\bigr)\in[0,\infty)^2
$
	with the following property: for every \(A,B\in\mathbb{M}_n\), there
	exists a unitary matrix \(V\in\mathbb{M}_n\) such that
	\[
	|A+B|
	\le
	C_1(n)\bigl(|A|+|B|\bigr)
	+
	C_2(n)V\bigl(|A|+|B|\bigr)V^*.
	\]
\end{prob}
Similarly, allowing two unitary conjugations on the right-hand side leads to the following natural variant.
\begin{prob}\label{prob:abs-two-unitries}
	Let \(n\ge 2\). Determine all pairs
	$
	\bigl(C_1(n),C_2(n)\bigr)\in[0,\infty)^2
	$
	with the following property: for every \(A,B\in\mathbb{M}_n\), there
	exist unitary matrices \(U, V\in\mathbb{M}_n\) such that
	\[
	|A+B|
	\le
	C_1(n)U\bigl(|A|+|B|\bigr)U^*
	+
	C_2(n)V\bigl(|A|+|B|\bigr)V^*.
	\]
\end{prob}
It follows easily from Thompson's inequality that every pair in $[1,\infty)^2$ has the property stated in Problem~\ref{prob:abs-two-unitries}.

More generally, allowing an average over \(n\) unitary conjugates, where
\(n\) is the size of the matrices, Bourin and Lee \cite{BL26a}
established the following comparison.

\vskip 5pt\noindent
\begin{thm}[\cite{BL26a}] If $A_1,\ldots, A_m\in\bM_n$ are normal matrices,  then,  for some unitary matrices $V_1,\ldots, V_n\in\bM_n$,
	$$
	\left|\sum_{k=1}^m A_k\right| \le \frac{1}{n}   \sum_{i=1}^nV_i\left\{\sum_{k=1}^m \left|A_k\right|\right\}V_i^*.
	$$
\end{thm}
\vskip 5pt\noindent

\subsection{Unitarily invariant norms}
	An interesting  result of Lee \cite{Lee12} gives the following bound for sums of matrices under any unitarily invariant norm (i.e., any symmetric norm, as we use both terminologies)  Zhang \cite{Zha26b} subsequently proved that this constant is sharp when \(n\ge m\).
 
\vskip 5pt\noindent
\begin{thm}[\cite{Lee12,Zha26b}]
	Let $A_1,\dots,A_m\in\mathbb M_n$. Then, for all symmetric norms,
	\begin{equation}\label{eq:Lee-uin}
		\Bigl\|\sum_{k=1}^m A_k\Bigr\|\ \le\ \sqrt{m}\,\Bigl\|\sum_{k=1}^m |A_k|\Bigr\|,
	\end{equation}
where the constant is sharp when $n\ge m$.
\end{thm}
\vskip 5pt\noindent
Zhang \cite{Zha26b} also established the following sharp estimate, which strengthens \eqref{eq:Lee-uin}.

\vskip 5pt\noindent
\begin{thm}[\cite{Zha26b}]\label{thm:Lee-optimal}
	Let $A_1,\dots,A_m\in\mathbb M_n$. Then,  for all symmetric norms,
	\begin{equation*}
		\Bigl\|\sum_{k=1}^m A_k\Bigr\|\ \le\ \sqrt{\min \{m,n\}}\,\Bigl\|\sum_{k=1}^m |A_k|\Bigr\|.
	\end{equation*}
	Moreover, the constant $\sqrt{\min\{m,n\}}$ is sharp.
\end{thm}
\vskip 5pt\noindent
\subsection{Schatten $p$-norms}
Lee \cite{Lee10} posed the problem of determining the optimal constant $c_p$ such that
\[
\|A+B\|_p\le c_p\,\||A|+|B|\|_p,\qquad \text{for all } A,B\in \mathbb M_n,
\]
which is now commonly referred to as \emph{Lee's problem}.
Lee's problem is notoriously difficult. The cases \(p=1\) and \(p=\infty\) are relatively easy to handle, whereas even the case \(p=2\) is already quite challenging; the remaining cases are still open. In particular, Lee \cite{Lee10} conjectured that
$c_2=\sqrt{\frac{1+\sqrt{2}}{2}}$.
This conjecture was first confirmed by Lin and Zhang \cite{LZ22}.
Subsequently, Zhang \cite{Zha25a} provided an alternative proof, and \cite{XZ25} further refined the estimate by deriving a sharper bound in which the constant is expressed more precisely in terms of the singular values of $A$ and $B$.

Motivated by Lee's problem, Tang and Zhang \cite{TZ26} proposed the following generalization for several matrices $A_k\in \mathbb M_n, 1\le k\le m$:
determine the optimal constant $c_p(m)$ depending only on $p$ and $m$ such that
\begin{equation}\label{eq:tang-zhang}
	\Bigl\lVert  \sum_{k=1}^m A_k\Bigr\rVert_p
	\ \le\ c_p(m)\,
	\Bigl\lVert \sum_{k=1}^m \bigl|A_k\bigr| \Bigr\rVert_p .
\end{equation}
For the above problem, Tang and Zhang~\cite{TZ26} stated several sharp values and general bounds. More precisely, they asserted that
\begin{enumerate}
	\item $c_{1}(m)=1$;
	\item $c_{2}(m)=\sqrt{\frac{1+\sqrt{m}}{2}}$, and they exhibited $m$ matrices of size $m\times m$ for which equality is attained in \eqref{eq:tang-zhang};
	\item $c_{\infty}(m)=\sqrt{m}$, and they exhibited $m$ matrices of size $m\times m$ for which equality is attained;
	\item $c_{p}(m)\le m^{\frac12-\frac{1}{2p}}$ for all $1\le p\le\infty$. 
\end{enumerate}
Since the constant $c_p(m)$ in \eqref{eq:tang-zhang} is required to work uniformly over the matrix size, the $m\times m$ extremizers in~\cite{TZ26} are sufficient to prove the optimality of any such dimension-independent constant. They do not, however, determine the best constant in each fixed dimension $n<m$. This motivates introducing the dimension-sensitive constant $c_p(m,n)$ and studying the sharper problem stated below.
\begin{prob}[\cite{Zha26b}]\label{prob:general-lee-problem}
	Let $1\le p\le \infty$ and let $A_1,\dots,A_m\in\mathbb M_n$. Determine the optimal constant $c_p^{\mathrm{abs}}(m,n)$ such that
	\[
	\Bigl\lVert  \sum_{k=1}^m A_k\Bigr\rVert_p
	\ \le\ c_p^{\mathrm{abs}}(m,n)\,
	\Bigl\lVert \sum_{k=1}^m \bigl|A_k\bigr| \Bigr\rVert_p .
	\]
\end{prob}

For Problem~\ref{prob:general-lee-problem}, Zhang \cite{Zha26b} proved the following dimension-dependent conclusions:
\begin{enumerate}
	\item $c_{1}^{\mathrm{abs}}(m,n)=1$;
	\item $c_{2}^{\mathrm{abs}}(m,n)=\sqrt{\frac{1+\sqrt{\min\{m,n\}}}{2}}$;
	\item $c_{\infty}^{\mathrm{abs}}(m,n)=\sqrt{\min\{m,n\}}$;
	\item $c_{p}^{\mathrm{abs}}(m,n)\le \bigl(\min\{m,n\}\bigr)^{\frac12-\frac{1}{2p}}$ for all $1\le p\le\infty$.
\end{enumerate}
In particular, these bounds are dimension-sensitive and, when $m>n$, they strictly sharpen the corresponding estimates in Tang and Zhang~\cite{TZ26}.
\section{The quadratic symmetric modulus and  triangle inequalities} \label{s:quadratic}

For a matrix $Z\in\bM_n$, its left and right moduli $|Z^*|$ and $|Z|$ occur in a number of important operator inequalities. There is no compelling reason to favor either the left or the right modulus. We will see that the consideration of symmetrized versions, or means of these two moduli, leads to several interesting results. In this section, we discuss the quadratic symmetric modulus.

The  quadratic symmetric modulus of $Z\in\bM_n$ is defined as
$$
|Z|_{\qsym}:=\sqrt{\frac{|Z|^2+|Z^*|^2}{2}}.
$$
Let ${\mathrm{Re\,}}Z=(Z+Z^*)/2$ and ${\mathrm{Im\,}}Z=(Z-Z^*)/2i$ denote the real and imaginary parts of $Z\in\bM_n$ and note that
 $$|Z|_{\qsym}=\sqrt{|{\mathrm{Re\,}}Z|^2+ |{\mathrm{Im\,}}Z|^2}.$$
Hence the quadratic symmetric modulus appears to be a very natural extension of the modulus in $\bC$.

\subsection{Operator version}
The following result is a striking symmetrized version of Thompson's inequality.

\vskip 5pt\noindent
\begin{thm}[{\cite{Zha26a}}]\label{thZ}  Let $A,B\in\bM_n$. Then there exist two unitaries $U,V\in\bM_n$ such that
$$
|A+B|_{\qsym} \le U|A|_{\qsym}U^* +V|B|_{\qsym}V^*.
$$
\end{thm}
\vskip 5pt\noindent

We adopt a proof of Bourin and Lee in \cite{BL26e}.
\begin{proof}[Proof of Theorem~\ref{thZ}] Note that \eqref{TI} 
	is equivalent to the existence of a pair of contractions $K,L\in\bM_n$ such that
	\begin{equation}\label{eqthomp2}
		|A+B| \le K|A|K^* +L|B|L^*.
	\end{equation}
	Clearly $\eqref{TI} \Rightarrow  \eqref{eqthomp2}$. Conversely,  $K|A|K^*=XX^*$ with $X=K|A|^{1/2}$. So $$K|A|K^*=U|A|^{1/2}K^*K|A|^{1/2}U^*$$ for some unitary $U\in\bM_n$. Since $K^*K\le I$, we have $K|A|K^*\le U|A|U^*$. Similarly, $L|B|L^*\le V|B|V^*$ for some unitary $V$. Hence  $\eqref{eqthomp2} \Rightarrow  \eqref{TI}$.

	Now, given four matrices $A_1,A_2,B_1,B_2\in\bM_n$ and letting
	$$
	A=\begin{bmatrix} A_1&0 \\ A_2&0\end{bmatrix}, \qquad B=\begin{bmatrix} B_1&0 \\ B_2&0\end{bmatrix}
	$$
	Thompson's theorem for $A$ and $B$ reads as
	$$
	\begin{bmatrix} \sqrt{|A_1 + B_1|^2 +|A_2+B_2|^2} & 0 \\ 0&0
	\end{bmatrix} \le
	U_0 \begin{bmatrix}\sqrt{|A_1|^2 +|A_2|^2} & 0 \\ 0&0 \end{bmatrix}  U_0 ^*
	+
	V_0 \begin{bmatrix}\sqrt{|B_1|^2 +|B_2|^2} & 0 \\ 0&0 \end{bmatrix} V_0 ^*
	$$
	for some unitary matrices $U_0,V_0\in\bM_{2n}$. Thus
	$$
	\sqrt{|A_1 + B_1|^2 +|A_2+B_2|^2} \le K\sqrt{|A_1|^2 +|A_2|^2} K^* + L\sqrt{|B_1|^2 +|B_2|^2} L^*
	$$
	for some contractions $K,L\in\bM_n$. By the first part of the proof, we have 
	$$
	\sqrt{|A_1 + B_1|^2 +|A_2+B_2|^2}\le U\sqrt{|A_1|^2 + |A_2|^2}U^* + V\sqrt{|B_1|^2 + |B_2|^2}V^* 
	$$
	for some unitaries $U,V\in\bM_n$. Then, following Zhang, we pick $A_2=A_1^*$ and $B_2=B_1^*$ to get
	$$
	\sqrt{|A_1 + B_1|^2 +|(A_1+B_1)^*|^2}\le U\sqrt{|A_1|^2 + |A_1^*|^2}U^* + V\sqrt{|B_1|^2 + |B_1^*|^2}V^*.
	$$
	Dividing by $\sqrt{2}$,
	$$
	|A_1+B_1|_{\qsym} \le U |A_1|_{\qsym} U^* + V|B_1|_{\qsym}V^*.
	$$
	This proves Theorem~\ref{thZ}.
\end{proof}

An interesting  special case \cite{BL26c} is:

\vskip 5pt\noindent
\begin{cor}\label{symq1} Let $Z\in\bM_n$. Then, for some unitaries $U,V\in\bM_n$,
$$
 |Z|_{\qsym} \le U |{\mathrm{Re\,}}Z| U^* + V |{\mathrm{Im\,}}Z|V^*.
$$
\end{cor}
\vskip 5pt\noindent

Applying Weyl's inequalities, we infer:

\vskip 5pt\noindent
\begin{cor}\label{symq2} Let $Z\in\bM_n$. Then, for all $j,k=0,1,2,\ldots$ such that $1+j+k\le n$, 
$$
\lambda_{1+j+k}^{\downarrow}\left( |Z|_{\qsym}\right) \le \lambda_{1+j}^{\downarrow}(|{\mathrm{Re\,}}Z|) + \lambda_{1+k}^{\downarrow}(|{\mathrm{Im\,}}Z|).
$$
\end{cor}
\vskip 5pt\noindent

This is a sharp result, in the sense that we cannot take $U=V$, even up to a scalar multiple. Indeed:

 \vskip 5pt\noindent
\begin{prop} Fix $c>0$. There exists $C\in\bM_3$ such that
$$
\det \frac{|C|+|C^*|}{2} > c \det (|{\mathrm{Re\,}}C|+  |{\mathrm{Im\,}}C|).
$$
\end{prop}
\vskip 5pt\noindent

\begin{proof}
Take
$$
C_t=\begin{pmatrix} 0&2t&0 \\
0&0&2 \\
0&0&0
\end{pmatrix},\qquad t>0.
$$
Then
$$
\frac{|C_t|+|C_t^*|}{2}=\operatorname{diag}(t,t+1,1),
$$
and
$$
|{\mathrm{Re\,}}C_t|+|{\mathrm{Im\,}}C_t|
=\frac{2}{\sqrt{t^2+1}}\operatorname{diag}(t^2,t^2+1,1).
$$
Therefore
$$
\frac{\det \frac{|C_t|+|C_t^*|}{2}}
{\det(|{\mathrm{Re\,}}C_t|+|{\mathrm{Im\,}}C_t|)}
=\frac{(t+1)\sqrt{t^2+1}}{8t}\longrightarrow\infty,
$$
which proves the assertion.
\end{proof}

Thompson's inequality for quadratic symmetric modulus immediately induces a norm inequality for every symmetric norm.

\vskip 5pt\noindent
\begin{cor}
For every symmetric norm on $\bM_n$, the functional $X\mapsto \| |X|_{\qsym} \|$ is a weakly unitarily invariant norm.
\end{cor}
\vskip 5pt\noindent

Combining Theorem~\ref{thZ} with the Aujla--Bourin inequality
\eqref{eq:Auj-Bou}, we obtain the following consequence.
\vskip 5pt\noindent
\begin{thm} Let $A,B\in\bM_n$ and let $f:[0,\infty)\to[0,\infty)$ be concave. Then, there exist unitary matrices $U,V$ such that 
	$$f(|A+B|_{\mathrm{qsym}})\le  Uf(|A|_{\mathrm{qsym}})U^*+Vf(|B|_{\mathrm{qsym}})V^*.
	$$
\end{thm}
\vskip 5pt\noindent

The following condition-number estimates for the quadratic symmetric
modulus were proved in \cite{Zha26a}. We record them here
as quadratic symmetric analogues of the estimates in
Section~\ref{s:right}.
\vskip 5pt\noindent
\begin{prop}[\cite{Zha26a}]
	Let \(A_1,\dots,A_m\in \mathbb M_n\) be invertible matrices whose condition numbers are bounded by \(\omega\ge1\). Then
	\begin{equation*}
		\bigl|A_1+\cdots+A_m\bigr|_{\mathrm{qsym}}
		\le
		\frac{\omega+1}{2\sqrt{\omega}}\,
		\bigl(|A_1|_{\mathrm{qsym}}+\cdots+|A_m|_{\mathrm{qsym}}\bigr),
	\end{equation*}
	where the constant is sharp as a function of $\omega$.
\end{prop}
\vskip 5pt\noindent
\vskip 5pt\noindent
\begin{prop}[\cite{Zha26a}]
	Let \(A_1,\dots,A_m\in \mathbb M_n\) be invertible matrices whose condition numbers are $\kappa_1,\ldots,\kappa_m$.
	Then
	\[
	\bigl|A_1+\cdots+A_m\bigr|_{\mathrm{qsym}}
	\le
	\frac{\kappa_1+1}{2\sqrt{\kappa_1}}\,
	|A_1|_{\mathrm{qsym}}+\cdots+ 	\frac{\kappa_m+1}{2\sqrt{\kappa_m}}\,|A_m|_{\mathrm{qsym}},
	\]
	where the coefficient function
	$
	c(\kappa):=\frac{\kappa+1}{2\sqrt{\kappa}}
	$
	is individually sharp for each term: for every \(\kappa\ge 1\), the coefficient \(c(\kappa)\)
	in front of a term of condition number \(\kappa\) cannot be replaced by any smaller constant,
	even if all the other coefficients are kept unchanged.
\end{prop}
\vskip 5pt\noindent

\subsection{Unitarily invariant norms}

We next pass from operator inequalities to inequalities for unitarily invariant
norms. In this setting, one can compare the ordinary sum directly with the sum
of the quadratic symmetric moduli. The following result gives a sharp universal
constant.

	\vskip 5pt\noindent
\begin{thm}[{\cite{Zha26b}}]\label{thm:sum-vs-qsym}
	Let $A_1,\dots,A_m\in\mathbb{M}_n$. Then,  for all symmetric norms,
	\[
	\Bigl\lVert \sum_{k=1}^m A_k \Bigr\rVert
	\ \le\ \sqrt2\,
	\Bigl\lVert \sum_{k=1}^m \bigl|A_k\bigr|_{\mathrm{qsym}} \Bigr\rVert .
	\]
	Moreover, the constant $\sqrt2$ is sharp.
\end{thm}
\vskip 5pt\noindent
	The preceding theorem compares the ordinary sum with the sum of quadratic symmetric moduli. Applying the quadratic symmetric modulus to the total sum leads to the following dimension-dependent estimate.

\vskip 5pt\noindent
\begin{thm}[{\cite{Zha26b}}]\label{thm:qsym-Lee}
	Let $A_1,\dots,A_m\in\mathbb{M}_n$. Then,  for all symmetric norms,
	\[
	\Bigl\lVert \Bigl|\sum_{k=1}^m A_k\Bigr|_{\mathrm{qsym}} \Bigr\rVert
	\ \le\ \sqrt{\min \{m,2n\}}\,
	\Bigl\lVert \sum_{k=1}^m \bigl|A_k\bigr|_{\mathrm{qsym}} \Bigr\rVert .
	\]
\end{thm}
\vskip 5pt\noindent
Furthermore, we strongly suspect that the constant in Theorem~\ref{thm:qsym-Lee} is not optimal; in particular, we have not been able to find any examples that attain  equality. This leaves an open problem.
\subsection{Schatten $p$-norms}

For Schatten $p$-norms, the comparison between the ordinary sum and the sum of
quadratic symmetric moduli admits sharper constants than those available for
general unitarily invariant norms. In fact, the optimal behavior exhibits a
phase transition at \(p=2\): no loss is needed for \(1\le p\le2\), whereas a
dimension-free factor appears for \(p\ge2\).
\vskip 5pt\noindent
\begin{thm}[{\cite{Zha26b}}]\label{thm:schatten-sum-vs-qsym-p}
	Let $1\le p\le \infty$ and let $A_1,\dots,A_m\in\mathbb M_n$. Then
	\begin{align*}
		\Bigl\lVert \sum_{k=1}^m A_k \Bigr\rVert_p
		\ &\le\ \,
		\Bigl\lVert \sum_{k=1}^m \bigl|A_k\bigr|_{\mathrm{qsym}} \Bigr\rVert_p, \quad\quad\quad\,\, 1\le p\le 2;\\
		\Bigl\lVert \sum_{k=1}^m A_k \Bigr\rVert_p
		\ &\le\ 2^{\frac{1}{2}-\frac{1}{p}}\,
		\Bigl\lVert \sum_{k=1}^m \bigl|A_k\bigr|_{\mathrm{qsym}} \Bigr\rVert_p, \quad 2\le p\le \infty.
	\end{align*}
	Moreover, the constants $1$ and $ 2^{\frac{1}{2}-\frac{1}{p}}$ are sharp.
\end{thm}
\vskip 5pt\noindent

Although Theorem~\ref{thm:schatten-sum-vs-qsym-p} holds for the whole
Schatten scale, the Hilbert--Schmidt case \(p=2\) has a special flavor. In
this case the inequality admits a particularly elegant proof, based only on
two simple observations. We record these observations as lemmas before giving
the proof.
	\vskip 5pt\noindent
\begin{lem}\label{lemCS1}
	For all $A,B\in\mathbb M_n$,
	\begin{equation}\label{eq:trace-cs}
		|{\mathrm{Tr\,}}(A^*B)|^2 \le\tr(\absv{A}\,\absv{B})\,\tr(\absv{A^*}\,\absv{B^*}).
	\end{equation}
	Consequently,
	\begin{equation}\label{eq:trace-amgm}
		|\tr(A^*B)|\ \le\ \frac{\tr(\absv{A}\,\absv{B})+\tr(\absv{A^*}\,\absv{B^*})}{2}.
	\end{equation}
\end{lem}
\vskip 5pt\noindent

\begin{proof}
	Let $A=U\absv{A}$ and $B=V\absv{B}$ be polar decompositions.
	Then 
	\begin{align*}
		|\tr(A^*B)|&=|\tr(\absv{A}U^*V\absv{B})| \\
		&=|\tr\Big((\absv{A}^\frac{1}{2}U^*V\absv{B}^\frac{1}{2})(\absv{B}^\frac{1}{2}\absv{A}^\frac{1}{2})\Big)|\\
		&\le\left\| \absv{A}^\frac{1}{2}U^*V\absv{B}^\frac{1}{2}\right\| _2\left\| \absv{B}^\frac{1}{2}\absv{A}^\frac{1}{2}\right\| _2\text{ (By Cauchy-Schwarz inequality)}\\
		&=\left( \tr(\absv{A^*}\,\absv{B^*})\right)^\frac{1}{2} \left( \tr(\absv{A}\,\absv{B})\right)^\frac{1}{2} 
	\end{align*}
	which is equivalent to \eqref{eq:trace-cs}.
	
	The bound \eqref{eq:trace-amgm} follows from \eqref{eq:trace-cs} by the scalar AM--GM inequality.
\end{proof}
Lieb's theorem \cite{Lie73} asserts that the map $(X,Y)\mapsto \tr(X^\frac{1}{2}Y^\frac{1}{2})$ is jointly concave on $\mathbb M_n^+\times\mathbb M_n^+$; for a shorter proof, see also \cite[Chapter~4]{Bha07}. That is,
\vskip 5pt\noindent
\begin{lem}\label{lem:lieb}
	Let $X_1,X_2,Y_1,Y_2\in\mathbb M_n^+$. Then
	\begin{equation*}
		\tr\!\left(\left(\frac{X_1+X_2}{2}\right)^\frac{1}{2}\left(\frac{Y_1+Y_2}{2}\right)^\frac{1}{2}\right)
		\ \ge\ \frac{\tr(X_1^\frac{1}{2}Y_1^\frac{1}{2})+\tr(X_2^\frac{1}{2}Y_2^\frac{1}{2})}{2}.
	\end{equation*}
\end{lem}
\vskip 5pt\noindent

Next, we present a sharp inequality for the Frobenius norm. Although it is subsumed by Theorem~\ref{thm:schatten-sum-vs-qsym-p}, the argument in this special case is particularly short and elegant.

\begin{proof}[Proof of the case $p=2$ in Theorem~\ref{thm:schatten-sum-vs-qsym-p}]
	Clearly,
	\begin{equation}\label{eq:Pk-frob}
		\|\qsma{A_k}\|_2^2=\|A_k\|_2^2.
	\end{equation}
	For any $i\neq j$, by using Lemma \ref{lemCS1} and Lemma~\ref{lem:lieb} with
	$X_1=\absv{A_i}^2$, $X_2=\absv{A_i^*}^2$, $Y_1=\absv{A_j}^2$ and $Y_2=\absv{A_j^*}^2$, we have
	\begin{align*}
		|\tr(A_i^*A_j)|
		&\le \frac{\tr(\absv{A_i}\,\absv{A_j})+\tr(\absv{A_i^*}\,|A_j^*|)}{2}\\
		&\le \tr\!\left(\left(\frac{|A_i|^2+\absv{A_i^*}^2}{2}\right)^\frac{1}{2}\left(\frac{\absv{A_j}^2+|A_j^*|^2}{2}\right)^\frac{1}{2}\right)\\
		&=\tr\left(\qsma{A_i}\qsma{A_j}\right). 
	\end{align*}
	Hence,
	\begin{align*}
		\|\sum_{k=1}^m A_k\|_2^2
		&=\sum_{k=1}^m \|A_k\|_2^2+2\sum_{i<j}\Re\,\tr(A_i^*A_j) \\
		&\le\sum_{k=1}^m \|A_k\|_2^2+2\sum_{i<j}\left| \tr(A_i^*A_j)\right| \\
		&\le
		\sum_{k=1}^m \|\qsma{A_k}\|_2^2+2\sum_{i<j}\tr\left(\qsma{A_i}\qsma{A_j}\right)\\
		&=\Big\|\sum_{k=1}^m \qsma{A_k}\Big\|_2^2,
	\end{align*}
	where we used \eqref{eq:Pk-frob}. Taking square roots yields the desired result.
	
\end{proof}

The preceding theorem compares the ordinary sum \(\sum_k A_k\) with the sum
of the individual quadratic symmetric moduli. A more intrinsic triangle-type
question is obtained by applying the quadratic symmetric modulus to the total
sum itself. This leads to the following dimension-sensitive optimization
problem.
\begin{prob}[{\cite{Zha26b}}]\label{prob:schatten-qsym-vs-qsym-p}
	Let $A_1,\dots,A_m\in\mathbb M_n$. Determine the optimal constant $c_p^{\mathrm{qsym}}(m,n)$ such that
	\[
	\Bigl\lVert \bigl|\sum_{k=1}^m A_k\bigr|_{\mathrm{qsym}} \Bigr\rVert_p
	\ \le\ c_p^{\mathrm{qsym}}(m,n)\,
	\Bigl\lVert \sum_{k=1}^m \bigl|A_k\bigr|_{\mathrm{qsym}} \Bigr\rVert_p .
	\]
\end{prob}
For Problem~\ref{prob:schatten-qsym-vs-qsym-p}, Zhang \cite{Zha26b} obtained the following results:
\begin{enumerate}
	\item $c_{1}^{\mathrm{qsym}}(m,n)=1$;
	\item $c_{2}^{\mathrm{qsym}}(m,n)\le \sqrt{\frac{1+\sqrt{\min\{m,2n\}}}{2}}$;
	\item $c_{\infty}^{\mathrm{qsym}}(m,n)\le \sqrt{\min\{m,2n\}}$;
	\item $c_{p}^{\mathrm{qsym}}(m,n)\le \bigl(\min\{m,2n\}\bigr)^{\frac12-\frac{1}{2p}}$ for all $1\le p\le\infty$.
\end{enumerate}

\section{The arithmetic symmetric modulus and triangle inequalities}\label{s:symmetric}

We now turn from the quadratic symmetric modulus to its arithmetic analogue.
For a matrix \(Z\in\bM_n\), the polar decomposition
\[
Z=|Z^*|U=U|Z|
\]
shows that the right modulus \(|Z|\) and the left modulus \(|Z^*|\) are two
natural positive representatives of \(Z\). It is therefore natural to average
them and consider the arithmetic symmetric modulus
\[
|Z|_{\sym}:=\frac{|Z|+|Z^*|}{2}.
\]
This modulus is more closely tied to the two-sided structure of \(Z\) than the
ordinary right modulus, but it is also more delicate than the quadratic
symmetric modulus studied in Section~\ref{s:quadratic}. The results in this
section illustrate both phenomena: some Thompson-type inequalities survive
after suitable modifications, while the most direct analogue fails in general.
For related developments, see also \cite{BL26c}.

\subsection{Operator version} As shown in the preceding sections, Thompson's inequality holds for the
ordinary right modulus and for the quadratic symmetric modulus. It is therefore
natural to ask whether the same phenomenon persists for the arithmetic
symmetric modulus. The answer, however, is negative. 
	
	\vskip 5pt\noindent
\begin{prop}[\cite{Zha26a}]
		Let $A,B\in \mathbb{M}_n$.
		In general, there do not exist unitaries $U,V$ such that
		\[
		|A+B|_{\mathrm{sym}}\le U|A|_{\mathrm{sym}}U^*+V|B|_{\mathrm{sym}}V^* .
		\] Indeed, there exist $A,B\in\mathbb{M}_2$ such that
		\[
		\bigl\||A+B|_{\mathrm{sym}}\bigr\|_{\infty}
		>
		\bigl\||A|_{\mathrm{sym}}\bigr\|_{\infty}+\bigl\||B|_{\mathrm{sym}}\bigr\|_{\infty}.
		\]
	\end{prop}
\vskip 5pt\noindent
	Although the direct Thompson-type inequality fails, it is natural to ask whether it can be recovered by allowing a uniform multiplicative constant. 
The answer is negative, as recently observed by Matthew Colbrook \cite{Col26}.

\vskip 5pt\noindent
\begin{rem}
Fix \(n\ge 3\). There is no finite constant \(C(n)>0\) such that, for every \(A,B\in\mathbb{M}_n\), there exist unitary matrices \(U,V\in\mathbb{M}_n\) satisfying
\begin{equation*}\label{eq:weak-triangle-sym}
    |A+B|_{\mathrm{sym}}
    \le
    C(n)\left(
    U|A|_{\mathrm{sym}}U^*
    +
    V|B|_{\mathrm{sym}}V^*
    \right).
\end{equation*}
In fact, it suffices to establish this for \(n=3\) and for matrices \(A\) and \(B\) of rank one.
\end{rem}
\vskip 5pt\noindent

The condition-number estimates for the right modulus also yield corresponding
bounds for the arithmetic symmetric modulus. In particular, Theorems~\ref{thm:Lee-triangle-condition-number}
and~\ref{thm:Zhang-triangle-condition-number} immediately imply the following
two consequences.
\vskip 5pt\noindent
\begin{prop}
	 	Let \(A_1,\dots,A_m\in \mathbb M_n\) be invertible matrices whose condition numbers are bounded by \(\omega\ge1\). Then
	\begin{equation*}
	\bigl|A_1+\cdots+A_m\bigr|_{\mathrm{sym}}
		\le
		\frac{\omega+1}{2\sqrt{\omega}}\,
		\bigl(|A_1|_{\mathrm{sym}}+\cdots+|A_m|_{\mathrm{sym}}\bigr),
	\end{equation*}
	where the constant is sharp as a function of $\omega$.
\end{prop}
\vskip 5pt\noindent
\vskip 5pt\noindent
\begin{prop}
	Let \(A_1,\dots,A_m\in \mathbb M_n\) be invertible matrices whose condition numbers are $\kappa_1,\ldots,\kappa_m$.
	Then
	\[
	\bigl|A_1+\cdots+A_m\bigr|_{\mathrm{sym}}
	\le
	\frac{\kappa_1+1}{2\sqrt{\kappa_1}}\,
	|A_1|_{\mathrm{sym}}+\cdots+ 	\frac{\kappa_m+1}{2\sqrt{\kappa_m}}\,|A_m|_{\mathrm{sym}},
	\]
	where the coefficient function
	$
	c(\kappa):=\frac{\kappa+1}{2\sqrt{\kappa}}
	$
	is individually sharp for each term: for every \(\kappa\ge 1\), the coefficient \(c(\kappa)\)
	in front of a term of condition number \(\kappa\) cannot be replaced by any smaller constant,
	even if all the other coefficients are kept unchanged.
\end{prop}
\vskip 5pt\noindent

Even without invertibility assumptions, one can still recover an operator
comparison by averaging over unitary orbits. In this direction, Bourin and
Lee \cite{BL26a} obtained the following averaged unitary-orbit inequality.

\vskip 5pt\noindent
\begin{thm}[{\cite{BL26a}}]\label{thsymaverage} Let $A_1,\ldots, A_m\in\bM_n$. Then,  for some unitaries $V_1,\ldots, V_n\in\bM_n$,
	$$
	\left|\sum_{k=1}^m A_k\right|_{\sym} \le \frac{\sqrt{2}}{n}   \sum_{i=1}^nV_i\left\{\sum_{k=1}^m \left|A_k\right|_{\sym}\right\}V_i^*.
	$$
\end{thm}
\vskip 5pt\noindent

\subsection{Unitarily invariant norms}
Zhang \cite{Zha26b} established a direct sharp comparison between the sum $\sum_k A_k$ and the sum of the symmetric moduli.
\vskip 5pt\noindent
\begin{thm}[{\cite{Zha26b}}]\label{thm:sum-vs-sym}
	Let $A_1,\dots,A_m\in\mathbb M_n$. Then, for all symmetric norms,
	\[
	\Bigl\lVert \sum_{k=1}^m A_k \Bigr\rVert
	\ \le\ 2\,
	\Bigl\lVert \sum_{k=1}^m \bigl|A_k\bigr|_{\mathrm{sym}} \Bigr\rVert .
	\]
	Moreover, the constant $2$ is sharp.
\end{thm}
\vskip 5pt\noindent

Bourin and Lee \cite{BL26e} proved the following uniform triangle estimate
for the arithmetic symmetric modulus. The constant is known to be sharp in
all dimensions \(n\ge3\), while the sharpness in dimension \(n=2\) remains
open; see \cite{Zha26b}.

\vskip 5pt\noindent
\begin{thm}[{\cite{BL26e}}]\label{thsymnorm} Let $A_1,\dots, A_m\in\bM_n$. Then, for all symmetric norms,
$$
\left\|\left|\sum_{k=1}^m A_k\right|_{\sym}\right\| \le \sqrt{2} \left\|\sum_{k=1}^m \left|A_k\right|_{\sym}\right\|,
$$
where the constant is sharp for all $n\ge 3$.
\end{thm}
\vskip 5pt\noindent

\vskip 5pt\noindent
\begin{prop}\label{Zex}
	There exist rank-one matrices \(A,B\in\bM_3\) with \(\pm1\)-entries such that
	\[
	\bigl\| |A+B|_{\sym}\bigr\|_{\infty}
	=
	\sqrt{2}\,
	\bigl\| |A|_{\sym}+|B|_{\sym}\bigr\|_{\infty}.
	\]
\end{prop}
\vskip 5pt\noindent

For instance, take
\[
A=\begin{pmatrix}
	1&-1&-1\\
	1&-1&-1\\
	1&-1&-1
\end{pmatrix},
\qquad
B=\begin{pmatrix}
	-1&-1&1\\
	1&1&-1\\
	-1&-1&1
\end{pmatrix}.
\]
Both matrices have rank one. A direct computation gives
\[
|A|_{\sym}+|B|_{\sym}=2I_3,
\qquad
\bigl\| |A+B|_{\sym}\bigr\|_{\infty}=2\sqrt2.
\]
Thus the constant \(\sqrt2\) in Theorem~\ref{thsymnorm} is sharp in
dimension \(3\). Embedding this example as a direct summand shows sharpness
in all dimensions \(n\ge3\).

Since the symmetric modulus is invariant by a congruence with a permutation matrix and by multiplication by $-1$, there are several examples of $\pm1$-matrices satisfying  Proposition \ref{Zex}.

Taking traces in Theorem~\ref{thsymnorm} gives only the displayed constant $\sqrt2$. The usual trace-norm triangle inequality is recovered separately from
$$
{\mathrm{Tr}}\,|A|_{\sym}={\mathrm{Tr}}\,|A|.
$$
The same unitary-orbit argument also yields the following eigenvalue refinement.

\vskip 5pt\noindent
\begin{prop}[{\cite{BL26e}}]\label{coreig}  Let $A_1,\dots, A_m\in\bM_n$. If $j\ge0$ and $1+3j\le n$, then
$$
\lambda^{\downarrow}_{1+3j}\left(\left|\sum_{k=1}^m A_k\right|_{\sym}\right) \le \sqrt{2} \lambda^{\downarrow}_{1+j}\left(\sum_{k=1}^m \left|A_k\right|_{\sym}\right).
$$
\end{prop}
\vskip 5pt\noindent

To describe a sharper form of the preceding estimate, we introduce a structural condition on the polar factor.

\vskip 5pt\noindent
\begin{defin} We say that a matrix $S$ is  {\it polar Hermitian} if it has a polar decomposition $S=V|S|$ in which $V$ is a scalar multiple of a Hermitian unitary. 
\end{defin}
\vskip 5pt\noindent

The class includes all involutions, as the next example records.

\vskip 5pt\noindent
\begin{eg} If $S\in\bM_n$ is an involution (a symmetry, $S=S^{-1}$), then $S$ is a polar Hermitian matrix. In fact, the unitary part  of $S$ is Hermitian.
Here is a simple proof from \cite{BL26b}. Since $S=U|S|=|S^*|U$ and $S=S^{-1}$ we infer $U|S|=U^*|S^*|^{-1}$ and so $U=U^*$ as the polar decomposition of an invertible matrix is unique. The main result of \cite{BL26b} says much more on the structure of involutions. 
\end{eg}
\vskip 5pt\noindent

This observation allows the constant $\sqrt{2}$ in Theorem~\ref{thsymnorm} to be improved to $1$ for polar-Hermitian sums.

\vskip 5pt\noindent
\begin{thm}[{\cite{BL26e}}]\label{thsympolnorm}  Let $A_1,\dots, A_m\in\bM_n$  be such that their sum is polar Hermitian.  Then, for all symmetric norms,
$$
\left\| \left|\sum_{k=1}^m A_k\right|_{\sym}\right\| \le  \left\|\sum_{k=1}^m \left|A_k\right|_{\sym}\right\|.
$$
\end{thm}
\vskip 5pt\noindent

\subsection{Schatten $p$-norms}
We finally specialize the preceding norm inequalities to the Schatten scale.
For the arithmetic symmetric modulus, Zhang obtained a sharp comparison
between the ordinary sum and the sum of the individual symmetric moduli. The
optimal constant depends only on \(p\), and is independent of both the number
of summands and the matrix size.
\vskip 5pt\noindent
\begin{thm}[{\cite{Zha26b}}]\label{thm:schatten-sum-vs-sym-p}
	Let $1\le p\le \infty$ and let $A_1,\dots,A_m\in\mathbb M_n$. Then
	\[
	\Bigl\lVert \sum_{k=1}^m A_k \Bigr\rVert_p
	\ \le 2^{1-\frac{1}{p}}\,
	\Bigl\lVert \sum_{k=1}^m \bigl|A_k\bigr|_{\mathrm{sym}} \Bigr\rVert_p.
	\]
	Moreover, the constant $2^{1-\frac{1}{p}}$ is sharp.
\end{thm}
\vskip 5pt\noindent

The preceding theorem is a ``sum versus sum'' estimate. A more intrinsic
triangle-type problem is to apply the arithmetic symmetric modulus to the
total sum itself and then compare it with the sum of the individual symmetric
moduli. This leads to the following dimension-sensitive optimization problem.
\begin{prob}[\cite{Zha26b}]\label{prob:schatten-sym-vs-sym-p}
	Let $A_1,\dots,A_m\in\mathbb M_n$. Determine the optimal constant $c_p^{\mathrm{sym}}(m,n)$ such that
	\[
	\Bigl\lVert  \bigl|\sum_{k=1}^m A_k\bigr|_{\mathrm{sym}} \Bigr\rVert_p
	\ \le\ c_p^{\mathrm{sym}}(m,n)\,
	\Bigl\lVert \sum_{k=1}^m \bigl|A_k\bigr|_{\mathrm{sym}} \Bigr\rVert_p .
	\]
\end{prob}
In Problem~\ref{prob:schatten-sym-vs-sym-p}, Zhang \cite{Zha26b} proved that
\begin{enumerate}
	\item		$c_1^{\mathrm{sym}}(m,n)=1$;
	\item  $	c_\infty^{\mathrm{sym}}(m,n)=\sqrt2$ for $n\ge 3$;
	\item $c_p^{\mathrm{sym}}(m,n)\le \min\{2^{1-\frac{1}{p}}, \sqrt{2}\}$;
	\item $	c_p^{\mathrm{sym}}(m,n)\ge \left(\frac{2^{\frac{p}{2}}+2^{1-\frac{p}{2}}}{3}\right)^{\frac{1}{p}}$ for $n\ge 3$;
	in particular, the right-hand side increases from $1$ at $p=2$ to $\sqrt2$ as $p\to\infty$.
\end{enumerate}

\section{Some hybrid triangle inequalities and  further problems}\label{s:hybrid}

The preceding sections treated the right modulus, the quadratic symmetric
modulus, and the arithmetic symmetric modulus separately. We now turn to
inequalities in which these different moduli interact with one another. Such
hybrid estimates often refine the norm inequalities discussed above, and they
also lead naturally to questions involving weak majorization, weak
log-majorization, and unitary-orbit averages.

We begin with the basic comparison between the arithmetic and quadratic
symmetric moduli. By the operator concavity of the square-root function,
\begin{equation}\label{compsqs}
	|Z|_{\sym}\le |Z|_{\qsym}.
\end{equation}
Indeed, this is just the inequality
\[
\frac{|Z|+|Z^*|}{2}
\le
\left(\frac{|Z|^2+|Z^*|^2}{2}\right)^{1/2}.
\]
In fact, a stronger squared form holds. By the operator convexity of \(t^2\),
\[
\left( \frac{|Z|+|Z^*|}{2}\right)^2
\le
\frac{|Z|^2+|Z^*|^2}{2},
\]
or equivalently,
\begin{equation}\label{compsqs2}
	|Z|_{\sym}^2\le |Z|_{\qsym}^2.
\end{equation}
Since the square-root function is operator monotone, \eqref{compsqs2}
implies \eqref{compsqs}. The next proposition \cite{BL26d} records a geometric
interpretation of the stronger inequality \eqref{compsqs2}.

\vskip 5pt\noindent
\begin{prop}\label{ub1}
	Let  $Z\in \bM_n$ and let ${\mathcal{U}}$ be the (closed) unit ball of $\bC^n$. Then,
	$$
	|Z|_{\sym}({\mathcal{U}})\subset |Z|_{\qsym}({\mathcal{U}}).
	$$
\end{prop}
\vskip 5pt\noindent

\begin{proof} Thanks to \eqref{compsqs2} it suffices to show that, for $A,B\in\bM_n^+$, we have the equivalence
	$$
	A^2\le B^2 \iff A({\mathcal{U}}) \subset B({\mathcal{U}}).
	$$
	Since each side ensures that the range of $A$ is included in the range of $B$, restricting our operators to the range of $B$ if necessary, we may assume that $B$ is invertible. Then
	$$
	A^2\le B^2 \iff B^{-1}A^2B^{-1} \le I \iff \| B^{-1}A\|_{\infty} \le 1 \iff B^{-1}A({\mathcal{U}}) \subset{\mathcal{U}}, 
	$$
	equivalently, $A({\mathcal{U}}) \subset B({\mathcal{U}})$.
\end{proof}

Given $A,B\in\bM_n^+$, we write $A\prec_{\mathrm{wlog}}B$ to mean that the following $n$ inequalities hold:
$$
\prod_{j=1}^p\lambda_{j}^{\downarrow} (A) \le \prod_{j=1}^p\lambda_{j}^{\downarrow} (B), \qquad p=1,\ldots,n.
$$
It is well known that $A\prec_{\mathrm{wlog}} B\Rightarrow A\prec_w B$.

Clearly, Theorem~\ref{thm:sum-vs-qsym} is equivalent to
	$$
\left| \sum_{k=1}^m A_k\right| \prec_{\mathrm{w}} \sqrt{2} \sum_{k=1}^m |A_k|_{\qsym}, \qquad A_k\in\bM_n.
$$ The comparison in Theorem~\ref{thm:sum-vs-qsym} can in fact be strengthened
from weak majorization to weak log-majorization. The following result of
Bourin and Lee gives such a refinement.

\vskip 5pt\noindent
\begin{thm}[{\cite{BL26d}}]
	Let $A_1,\ldots,A_m\in\bM_n$. Then
	\[
	\left| \sum_{k=1}^m A_k\right|
	\prec_{\mathrm{wlog}}
	\sqrt{2}\sum_{k=1}^m |A_k|_{\qsym}.
	\]
\end{thm}
\vskip 5pt\noindent

This suggests that several norm inequalities discussed earlier may admit
stronger spectral forms. We therefore record the following open problem.

\begin{prob}
	Determine which of the unitarily invariant norm inequalities in
	Sections~\ref{s:right}--\ref{s:symmetric} can be strengthened to weak
	log-majorization.
\end{prob}

The same circle of ideas also leads to direct comparisons between the
quadratic and arithmetic symmetric moduli. In particular, Bourin and Lee
proved the following averaged unitary-orbit estimate.

\vskip 5pt\noindent
\begin{thm}[{\cite{BL26a}}]
	Let $Z\in \bM_n$. Then, for some unitaries $V_1,\ldots,V_n\in\bM_n$,
	\[
	|Z|_{\qsym}
	\le
	\frac{\sqrt{2}}{n}
	\sum_{i=1}^n V_i |Z|_{\sym}V_i^*.
	\]
	The constant $\sqrt{2}$ cannot be diminished.
\end{thm}
\vskip 5pt\noindent

The preceding results concern the two endpoint cases \(p=1\) and \(p=2\),
corresponding respectively to the arithmetic and quadratic symmetric moduli.
It is natural to interpolate between them by considering the following
power-mean symmetric modulus. For \(p>0\), define
\[
|Z|_{p,\mathrm{sym}}
:=
\left(\frac{|Z|^p+|Z^*|^p}{2}\right)^{1/p}.
\]

For this family, Zhang obtained a sharp range for the Cartesian
Thompson-type inequality. The result shows that the quadratic case is the
largest power for which such an estimate remains valid.

\vskip 5pt\noindent
\begin{thm}[{\cite{Zha26a}}]
	Fix \(n\ge2\) and \(p>0\). The following assertions are equivalent:
	\begin{enumerate}
		\item \(0<p\le2\);
		\item for every \(Z\in\mathbb M_n\), there exist unitaries
		\(U,V\in\mathbb M_n\) such that
		\[
		|Z|_{p,\mathrm{sym}}
		\le
		U|\Ree Z|U^*+V|\Imm Z|V^* .
		\]
	\end{enumerate}
\end{thm}
\vskip 5pt\noindent

This theorem suggests that the moduli \( |\cdot|_{p,\mathrm{sym}} \) may
exhibit substantially different triangle behavior in the ranges \(0<p\le2\)
and \(p>2\). We therefore leave the following problem for further study.

\begin{prob}
	Study triangle inequalities and optimal constants for the power-mean
	symmetric moduli \( |\cdot|_{p,\mathrm{sym}} \).
\end{prob}

In particular, the quadratic case suggests a broader norm-theoretic question.
Although Thompson-type operator inequalities may fail outside the range
\(0<p\le2\), it is still natural to ask whether the associated symmetric-norm
functionals satisfy a triangle inequality.

\begin{conj}
	Let \(p\ge1\), and let
	\[
	|X|_{p,\mathrm{sym}}
	:=
	\left(\frac{|X|^p+|X^*|^p}{2}\right)^{1/p}.
	\]
	Then, for every symmetric norm \(\|\cdot\|\) on \(\mathbb M_n\), the
	functional
	\[
	X\mapsto \bigl\| |X|_{p,\mathrm{sym}}\bigr\|
	\]
	is a weakly unitarily invariant norm on \(\mathbb M_n\).
\end{conj}


\begin{thebibliography}{99}

	
\bibitem[AB07]{AB07}
J.~S.~Aujla and J.-C.~Bourin,
\emph{Eigenvalue inequalities for convex and log-convex functions},
Linear Algebra Appl. \textbf{424} (2007), no.~1, 25--35.
\doi{10.1016/j.laa.2006.02.027}

\bibitem[And89]{And89}
T.~Ando,
\emph{Majorization, doubly stochastic matrices, and comparison of eigenvalues},
Linear Algebra Appl. \textbf{118} (1989), 163--248.
\doi{10.1016/0024-3795(89)90580-6}

\bibitem[Bha97]{Bha97}
R.~Bhatia,
\emph{Matrix Analysis},
Graduate Texts in Mathematics, vol.~169, Springer-Verlag, New York, 1997.
\doi{10.1007/978-1-4612-0653-8}

	\bibitem[Bha07]{Bha07}
R.~Bhatia,
\emph{Positive Definite Matrices},
Princeton Series in Applied Mathematics,
Princeton University Press, Princeton and Oxford, 2007.

\bibitem[Bou99]{Bou99}
J.-C.~Bourin,
\emph{Some inequalities for norms on matrices and operators},
Linear Algebra Appl. \textbf{292} (1999), no.~1--3, 139--154 \doi{10.1016/S0024-3795(99)00024-5}.

\bibitem[Bou06]{Bou06}
J.-C.~Bourin,
\emph{Matrix versions of some classical inequalities},
Linear Algebra Appl. \textbf{416} (2006), no.~2--3, 890--907.
\doi{10.1016/j.laa.2006.01.002}

\bibitem[Bou10]{Bou10}
J.-C.~Bourin,
\emph{A matrix subadditivity inequality for symmetric norms},
Proc. Amer. Math. Soc. \textbf{138} (2010), no.~2, 495--504.
\doi{10.1090/S0002-9939-09-10103-X}

\bibitem[BL12]{BL12}
J.-C.~Bourin and E.-Y.~Lee,
\emph{Unitary orbits of Hermitian operators with convex or concave functions},
Bull. Lond. Math. Soc. \textbf{44} (2012), no.~6, 1085--1102.
\doi{10.1112/blms/bds080}

\bibitem[BL26a]{BL26a}
J.-C.~Bourin and E.-Y.~Lee,
\emph{Averages over matrix unitary orbits and spectral order},
preprint, arXiv:2606.15624, 2026.
\doi{10.48550/arXiv.2606.15624}

\bibitem[BL26b]{BL26b}
J.-C.~Bourin and E.-Y.~Lee,
\emph{Involutions and angles between subspaces},
Linear Algebra Appl. \textbf{730} (2026), 124--137.
\doi{10.1016/j.laa.2025.10.017}

\bibitem[BL26c]{BL26c}
J.-C.~Bourin and E.-Y.~Lee,
\emph{Matrix parallelogram laws and symmetric moduli},
Internat. J. Math. \textbf{37} (2026), no.~3, 2650018.
\doi{10.1142/S0129167X26500187}

\bibitem[BL26d]{BL26d}
J.-C.~Bourin and E.-Y.~Lee,
\emph{Some hybrid matrix triangle inequalities},
Linear Algebra Appl. \textbf{749} (2026), 197--207.
\doi{10.1016/j.laa.2026.06.033}

\bibitem[BL26e]{BL26e}
J.-C.~Bourin and E.-Y.~Lee,
\emph{Triangle inequalities for the operator symmetric modulus},
Proc. Amer. Math. Soc. 155 (2027), 261--270 \doi{10.1090/proc/17861}. 


\bibitem[Col26]{Col26}
 M.~J.~Colbrook, MI-06, Thompson-type domination for the arithmetic symmetric modulus, Open
 		Problems in Numerical Linear Algebra,  2026,   github.com/ajt60gaibb/OpenProblemsInNLA/tree/main/matrix-inequalities-and-norms/MI-06.

\bibitem[FH55]{FH55}
K.~Fan and A.~J. Hoffman,
\emph{Some metric inequalities in the space of matrices},
Proc. Amer. Math. Soc. \textbf{6} (1955), 111--116.
\doi{10.1090/S0002-9939-1955-0067841-7}

\bibitem[Gol65]{Gol65}
S.~Golden,
\emph{Lower bounds for the Helmholtz function},
Phys. Rev. (2) \textbf{137} (1965), B1127--B1128.
\doi{10.1103/PhysRev.137.B1127}

\bibitem[HJ91]{HJ91}
R.~A.~Horn and C.~R.~Johnson,
\emph{Topics in Matrix Analysis},
Cambridge University Press, Cambridge, 1991.
\doi{10.1017/CBO9780511840371}

\bibitem[HP14]{HP14}
F.~Hiai and D.~Petz,
\emph{Introduction to Matrix Analysis and Applications},
Universitext, Springer, Cham; Hindustan Book Agency, New Delhi, 2014.
\doi{10.1007/978-3-319-04150-6}

\bibitem[Lee10]{Lee10}
E.-Y.~Lee,
\emph{Rotfel'd type inequalities for norms},
Linear Algebra Appl. \textbf{433} (2010), no.~3, 580--584.
\doi{10.1016/j.laa.2010.03.029}

\bibitem[Lee12]{Lee12}
E.-Y.~Lee,
\emph{How to compare the absolute values of operator sums and the sums of absolute values?},
Oper. Matrices \textbf{6} (2012), no.~3, 613--619.
\doi{10.7153/oam-06-42}

	\bibitem[Lie73]{Lie73}
E.~H.~Lieb,
\emph{Convex trace functions and the Wigner--Yanase--Dyson conjecture},
Adv. Math. \textbf{11} (1973), 267--288.
\doi{10.1016/0001-8708(73)90011-X}

\bibitem[LZ22]{LZ22}
J.~Lin and Y.~Zhang,
\emph{A proof of Lee's conjecture on the sum of absolute values of matrices},
J. Math. Anal. Appl. \textbf{516} (2022), no.~2, Paper No.~126542, 6~pp.
\doi{10.1016/j.jmaa.2022.126542}

\bibitem[Neu27]{Neu27}
J.~von Neumann,
\emph{Thermodynamik quantenmechanischer Gesamtheiten},
Nachr. Ges. Wiss. G\"ottingen Math.-Phys. Kl. \textbf{1927} (1927), 273--291.

\bibitem[Neu37]{Neu37}
J.~von Neumann,
\emph{Some matrix-inequalities and metrization of matric-space},
Tomsk Univ. Rev. \textbf{1} (1937), 286--300.

\bibitem[Rot69]{Rot69}
S.~Yu.~Rotfel'd,
\emph{The singular numbers of the sum of completely continuous operators},
in \emph{Spectral Theory}, Topics in Mathematical Physics, vol.~3,
Consultants Bureau, New York, 1969, pp.~73--78.
\doi{10.1007/978-1-4684-7589-0_5}

\bibitem[TZ26]{TZ26}
Q.~Tang and S.~Zhang,
\emph{Generalizing Lee's conjecture on the sum of absolute values of matrices},
Linear Algebra Appl. \textbf{731} (2026), 196--204.
\doi{10.1016/j.laa.2025.11.015}

\bibitem[Tho65]{Tho65}
C.~J. Thompson,
\emph{Inequality with applications in statistical mechanics},
J. Math. Phys. \textbf{6} (1965), no.~11, 1812--1813.
\doi{10.1063/1.1704727}

\bibitem[Tho76]{Tho76}
R.~C. Thompson,
\emph{Convex and concave functions of singular values of matrix sums},
Pacific J. Math. \textbf{66} (1976), no.~1, 285--290.
\doi{10.2140/pjm.1976.66.285}


\bibitem[XZ25]{XZ25}
Z.~Xu and T.~Zhang,
\emph{Sharp perturbation bounds on the Frobenius norm of
subunitary and positive polar factor},
preprint, arXiv:2507.14940v2, 2025.
\doi{10.48550/arXiv.2507.14940}


\bibitem[Zha25a]{Zha25a}
T.~Zhang,
\emph{A new proof of Lee's conjecture on the Frobenius norm via the matrix Cauchy--Schwarz inequality},
Linear Algebra Appl. \textbf{725} (2025), 355--358.
\doi{10.1016/j.laa.2025.07.017}


\bibitem[Zha26a]{Zha26a}
T.~Zhang,
\emph{An operator triangle inequality for the quadratic symmetric modulus},
preprint, arXiv:2602.01463, 2026.
\doi{10.48550/arXiv.2602.01463}

\bibitem[Zha26b]{Zha26b}
T.~Zhang,
\emph{Operator symmetric moduli and sharp triangle inequalities}, J. London Math. Soc. \textbf{114} (2026), Journal London Math. soc. no. 2, Art. e70672.
\doi{10.1112/jlms.70672}.


\end{thebibliography}
\end{document}